\documentclass[a4paper, 11pt]{article}

\usepackage{xcolor}
\usepackage{amsmath,amsfonts,amssymb} % for math
\usepackage{enumitem}
\usepackage{amsthm} % theorems
\usepackage{bbm}
\usepackage{graphicx,caption,subcaption} % figures and subfigures
\usepackage{fullpage}

\allowdisplaybreaks % otherwise the align environment is stupid.

\begingroup
    \makeatletter
    \@for\theoremstyle:=definition,remark,plain\do{%
        \expandafter\g@addto@macro\csname th@\theoremstyle\endcsname{%
            \addtolength\thm@preskip\parskip
            }%
        }
\endgroup

\newtheorem{thm}{Theorem}[section]
\newtheorem{cor}[thm]{Corollary}

\newtheorem{rem}[thm]{Remark}
\newtheorem{claim}[thm]{Claim}
\newtheorem{lemma}[thm]{Lemma}

\newtheorem{definition}[thm]{Definition}

\newtheorem{subclaim}{Subclaim}

\begin{document}

\title{\vspace{-0.5in} Ore-type conditions for existence of a jellyfish in a graph}

\author{{{Jaehoon Kim}}\thanks{
\footnotesize {KAIST, Daejeon, South Korea. E-mail: \texttt {jaehoon.kim@kaist.ac.kr}.
Research is supported by the Fulbright Visiting Scholar Fellowship and by the National Research Foundation of Korea (NRF) grant funded by the Korean government(MSIT) No. RS-2023-00210430. 
}}
\and
{{Alexandr Kostochka}}\thanks{
\footnotesize {University of Illinois at Urbana--Champaign, Urbana, IL 61801.
 %and Sobolev Institute of Mathematics, Novosibirsk 630090, Russia. 
 E-mail: \texttt {kostochk@illinois.edu}.
 Research %%% of this author
is supported in part by  NSF  Grant DMS-2153507 and by Campus Research Board Award RB24000 of the University of Illinois Urbana-Champaign.
}}
\and
{{Ruth Luo}}\thanks{
\footnotesize {University of South Carolina, Columbia, SC 29208. E-mail: \texttt {ruthluo@sc.edu}.
}}
}

\date{ \today}
\maketitle

\vspace{-0.3in}

\begin{abstract}
The famous Dirac's Theorem states that for each $n\geq 3$ every $n$-vertex graph $G$ with minimum degree 
$\delta(G)\geq n/2$ has  a hamiltonian cycle. When $\delta(G)< n/2$, this cannot be guaranteed, but the existence of some other specific subgraphs can be provided. Gargano,  Hell,  Stacho and  Vaccaro proved that every connected $n$-vertex graph $G$ with $\delta(G)\geq (n-1)/3$ contains a spanning {\em spider}, i.e., a spanning tree with at most one vertex of degree at least $3$. Later, Chen,  Ferrara, Hu,   Jacobson and  Liu proved the stronger (and exact) result that for $n\geq 56$ every connected $n$-vertex graph $G$ with $\delta(G)\geq (n-2)/3$ contains a spanning {\em broom}, i.e.,  a spanning spider obtained by joining the center of a star to an endpoint of a path. They also showed that a $2$-connected graph $G$ with $\delta(G)\geq (n-2)/3$ and some additional properties contains a spanning {\em jellyfish} which is a graph obtained by gluing the center of a star to a vertex in a cycle disjoint from that star. Note that every spanning jellyfish contains a spanning broom.

The goal of this paper is to prove an exact Ore-type bound which guarantees the existence of a spanning jellyfish: We prove that if $G$ is a $2$-connected graph on $n$ vertices such that every non-adjacent pair of vertices $(u,v)$ satisfies $d(u) + d(v) \geq \frac{2n-3}{3}$, then $G$ has a spanning jellyfish. As  corollaries, we obtain strengthenings   of two results by Chen et al.: a minimum degree condition guaranteeing the existence of a spanning jellyfish, and an Ore-type sufficient condition for the existence of a spanning broom. The corollaries are sharp for infinitely many $n$.
One of the main ingredients of our proof is a modification of the Hopping Lemma due to Woodall.

%
%The goal of this paper is to prove that for $n\geq 13$, every $2$-connected $n$-vertex graph $G$ with $\delta(G)\geq (n-1)/3$ contains a spanning jellyfish. The bound $(n-1)/3$ is exact. Moreover, we prove the somewhat stronger bound of Ore-type: For $n\geq 13$, every $2$-connected $n$-vertex graph $G$ with $\sigma_2(G)\geq (2n-3)/3$ contains a spanning jellyfish, where $\sigma_2(G)$ is the minimum sum of degrees of two non-adjacent vertices in $G$. This also yields a corresponding Ore-type condition for the existence of a spanning broom in a connected $n$-vertex graph.

\end{abstract}

\section{Introduction}

\subsection{Terminology and known results}

The {\em degree} $d_G(v)$ of a vertex $v$ in a graph $G$ is the number of edges containing $v$. The {\em minimum degree}, $\delta(G)$, is the minimum over degrees of all vertices of $G$, and the  {\em maximum degree}, $\Delta(G)$, is the corresponding maximum. 
%By $\sigma_2(G)$ we denote the minimum of the sum $d_G(u)+d_G(v)$ over all pairs $(u,v)$ of nonadjacent vertices in $G$. 
 When there is no confusion, we often drop subscripts and  write $d(v)$ for $d_G(v)$. 
%The {\em circumference}, $c(G)$, is the length of a longest cycle in $G$.

A {\em hamiltonian cycle} (respectively, path)  in a graph is a cycle  (respectively, path)  that visits every vertex. %We say a graph is {\em hamiltonian} if it contains a hamiltonian cycle. 
 Sufficient conditions for existence of hamiltonian cycles and paths in graphs have been well studied. 
A central result in graph theory is Dirac's Theorem~\cite{D} which states that for $n\geq 3$, every $n$-vertex graph $G$ with 
$\delta(G)\geq n/2$ has a hamiltonian cycle and every $n$-vertex graph $G$ with 
$\delta(G)\geq (n-1)/2$ has a hamiltonian path.

Generalizing Dirac’s theorem, the minimum degree conditions guaranteeing spanning structures have been investigated for extensive list of graphs, such as spanning trees with bounded maximum degree, powers of Hamilton cycles, graphs with small bandwidths and so on. Note that the required minimum degree conditions for all these results are at least $(n-1)/2$ as the disjoint union $K_{\lfloor n/2 \rfloor}\cup K_{\lceil n/2 \rceil}$ is a disconnected graph with minimum degree $\lfloor n/2 \rfloor - 1$ that contains no connected spanning subgraph.

The bottleneck in this example is the connectivity rather than the minimum degree. To better analyze the relation between the minimum degree and the substructures of graphs, it is natural to ask for minimum degree conditions on $G$ that guarantee some spanning subgraph $H$ provided $G$ also satisfies some given connectivity conditions. For instance, it is necessary for $G$ to be $1$-connected and $2$-connected in the cases when $H$ is a spanning tree and a spanning cycle, respectively.

%This makes sense for the cases where the disjoint union of cliques provides an (almost) tight examples for lacking a copy of $H$, i.e. when $H$ is a graph close to being a tree or a cycle.

In the case where $H$ is a spanning tree and $G$ is connected, a number of results were proved for special classes of trees showing that minimum degrees that are strictly less than $(n-1)/2$ guarantee those trees. For example, Win~\cite{Win} proved that every $n$-vertex  connected
graph with $\delta(G)\geq (n-1)/k$ contains a spanning tree $T$ with $\Delta(T)\leq k$. Broersma and 
Tuinstra~\cite{BT} showed that every $n$-vertex  connected
graph with $\delta(G)\geq (n-k+1)/2$ contains a spanning tree  with 
 at most $k$ leaves. Several results have been obtained on the existence of spanning trees with bounded number of {\em branching vertices}, that is, vertices of degree at least $3$, see e.g.~\cite{CS,FKKLR,
  GHSV1,GHSV2,OY}. 
  
  Trees with at most one branching vertex are called {\em spiders}. Gargano,  Hell,  Stacho and  Vaccaro~\cite{GHSV2} proved that every connected $n$-vertex graph $G$ with $\delta(G)\geq (n-1)/3$ contains a spanning { spider}. Chen,  Ferrara, Hu,   Jacobson and  Liu~\cite{CFHJL} showed that the same condition implies existence of a more restricted kind of spiders. A {\em broom} is
   a spanning spider obtained by joining the center of a star to an endpoint of a path. In other words,
   a  {\em broom} is a spider with all but at most one leg of length $1$.

\begin{thm}[\cite{CFHJL}]\label{Chen}
If $G$ is a connected graph of order $n\geq 56$ with $\delta(G)\geq \frac{n-2}{3}$, then $G$ contains
a spanning broom. The condition $\delta(G)\geq \frac{n-2}{3}$ is sharp.
\end{thm}

For the case where $H$ is a cycle, Dirac~\cite{D} proved that the every $2$-connected graph $G$ contains a cycle of length at least $2\delta(G)$. 
Chen et al.~\cite{CFHJL} also considered a spanning structure similar to brooms in $2$-connected graphs---a {\bf jellyfish} consists of a cycle $C$ and a set of vertices $X\subseteq V(G) - V(C)$ all adjacent to the same vertex in $C$, see Fig. 1.
Observe that a jellyfish can be obtained from a broom by the addition of one edge. See~\cite{keyring1, keyring2, keyring3} for other results related to jellyfish (also called {\em keyrings}).

%Chen at al.~\cite{CFHJL} also considered the existence of spanning jellyfish in a $2$-connected graph. Note that it is easy to check that $2$-connectedness is required as in Remark~\ref{rmk:1}.

\begin{figure}[h]
    \centering
    \begin{minipage}{0.45\textwidth}
        \centering        \includegraphics[width=0.7\textwidth]{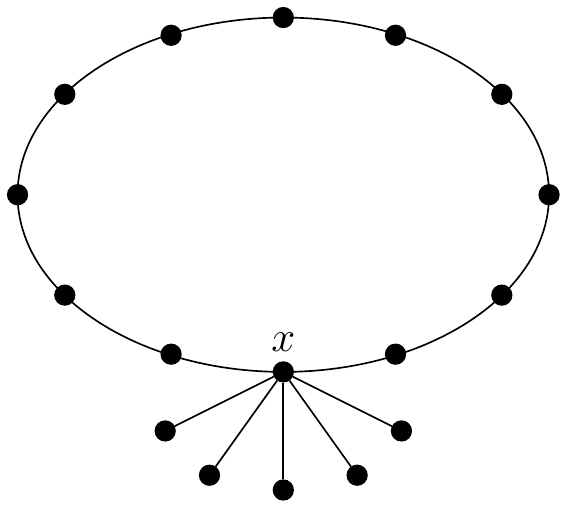} 

       \caption{A jellyfish graph}
        \label{jellypic}
        
    \end{minipage}
    \begin{minipage}{0.5\textwidth}
        \centering
        \includegraphics[width=0.7\textwidth]{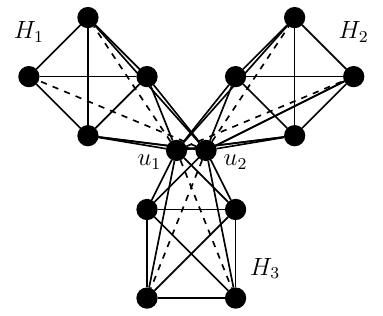} 
        \caption{A graph with no spanning jellyfish}
        \label{ex2}
    \end{minipage}
\end{figure}

 Let $c(G)$ denote the length of a longest cycle in $G$ and $p(G)$ denote the number of vertices in a longest path in $G$. 

\begin{thm}[\cite{CFHJL}]\label{Chen2}
Let  $G$ be a $2$-connected graph of order $n$ such that $\delta(G)\geq \frac{n-2}{3}$ and
$p(G)-c(G)\geq 1$. 
 Then $G$ contains
a spanning jellyfish. 
\end{thm}

Here, $2$-connectedness is required, see Remark~\ref{rmk:1}.
Note that the inequality $p(G)-c(G)\geq 1$ is a serious restriction. In particular,
if we drop this restriction with no other changes, then the conclusion of  Theorem~\ref{Chen2} would not hold:
 there are 
 $2$-connected graphs of order $n$ with $\delta(G)\geq \frac{n-2}{3}$ that do not contain a spanning jellyfish.
 One example $H$ for $n=3s+2$ is obtained from $3$ disjoint copies $H_1,H_2,H_3$  of the complete graph $K_s$ by adding two nonadjacent vertices, say $u_1,u_2$ and making each of $u_1,u_2$ adjacent to all but one vertex in each of $H_1,H_2,H_3$ so that every vertex in $V(H_1)\cup V(H_2)\cup V(H_3)$ is adjacent to at least one of 
 $u_1,u_2$; see Fig. 2. Indeed, $\delta(H) = s = (n-2)/3$, each cycle in $H$ intersects at most two $H_i$'s, and the vertices of the remaining $H_j$ do not have a single common neighbor. 
 
 In this paper, we give an exact degree condition for a $2$-connected $n$-vertex graph to guarantee the existence of a spanning jellyfish. Moreover, we prove a slightly stronger bound of Ore-type. For a graph $G$,  define $\sigma_2(G) = \min_{uv \notin E(G)} d(u) +d(v)$. Our main result is:

\begin{thm}\label{main}
For $n \geq 13$, every  $2$-connected $n$-vertex graph $G$ with $\sigma_2(G) \geq \frac{2n-3}{3}$  contains a spanning jellyfish.
\end{thm}

This theorem implies a similar result with a bound on the minimum degree.
\begin{cor}\label{maindirac} For $n \geq 13$, every  $2$-connected $n$-vertex graph $G$ with 
$\delta(G) \geq \frac{n-1}{3}$ contains a spanning jellyfish.
\end{cor}

Note that the graph $H$ in Fig. 2  %defined a paragraph above Theorem~\ref{main} 
shows the sharpness of both,
Theorem~\ref{main} and Corollary~\ref{maindirac}.

From Theorem~\ref{main} we derive a slight strengthening for $n\neq 2 \pmod 3$
of Theorem~\ref{Chen}:
%a result by~\cite{CFHJL}] on the existence of a spanning broom.

\begin{thm}\label{broom}Let $n \geq 13$. If $G$ is a connected, $n$-vertex graph with $\sigma_2(G) \geq \frac{2n-3}{3}$, then $G$ contains a spanning broom.
\end{thm}

\begin{rem}\label{rmk:1}
	The property of $2$-connectedness is necessary in Theorem~\ref{main} and Corollary~\ref{maindirac}: the connected graph $G$ consisting of disjoint cliques of order $\lfloor n/2 \rfloor$ and $\lceil n/2 \rceil$ with a single additional edge has $\delta(G)=\lfloor n/2\rfloor-1$, $\sigma_2(G) = n-2$ and no spanning jellyfish.
\end{rem}

%For $n=2 ({\rm mod}\, 4)$, there are $\frac{n-2}{4}$-connected graphs not containing spanning spiders. Thus, $g(n)\geq \frac{n-1}{4}$.

%{\bf Remark 2}. The bound in Theorem~\ref{main} is best possible for infinitely many $n$ due to the following construction. 
%Let $n = 2 \pmod{3}$, $n \geq 8$, and take $3$ copies $H_1, H_2, H_3$ of $K_{(n+4)/3}$ all sharing two common vertices $u$ and $v$. In each $H_i$, remove one edge incident to $u$ and one edge incident to $v$ such that these edges do not share a common endpoint. Call this graph $H$.
%We have $\delta(H) = (n+4)/3 - 2 = (n-2)/3$, $\sigma_2(H) = \frac{2n-4}{3}$, any cycle in $H$ intersects at most two $H_i$'s outside of $\{u,v\}$, and the vertices of the remaining $H_j$ do not have a single common neighbor. 

\bigskip
The structure of the paper is as follows. In the next section, we introduce some notation and prove or cite some claims for graphs with ``high" Ore-degree, $\sigma_2(G)$. In Section~\ref{seca} we consider a minimum counter-example $G$ to Theorem~\ref{main} and show that all components of a specially chosen longest cycle in $G$ are single vertices. In Section~\ref{secb} we finish the proof of Theorem~\ref{main}. An important tool in this section is a modification of the $50$-year-old Hopping Lemma by Woodall~\cite{W}. We prove Theorem~\ref{broom} in Section~\ref{secc}.

\section{Preliminaries}

For a graph $G$, $|G|$ denotes the number of vertices of $G$.
%For two graphs $G$ and $H$, we use $G\cap H$ to denote the vertex set $V(G)\cap V(H)$ and $G-H$ to denote the induced subgraph $G[V(G)\setminus V(H)]$. 
For a statement $S$, $\mathbbm{1}_{S}$ is an indicator function which is $1$ if $S$ is true and $0$ otherwise. 
For sets $A$ and $B$, we define $I_A=\mathbbm{1}_{A\neq \emptyset}$ and $J_{A,B} = I_{A\setminus B} + I_{B\setminus A}$.

\begin{lemma}\label{neighbor2}
Let $C=v_1 \dots v_c$ be a cycle and $2\leq q\leq c/2$. Suppose that $A$ and $B$ are nonempty subsets of nonconsecutive vertices of $C$ with $|A\cup B|\geq 2$ such that the distance between $v_i$ and $v_j$ in $C$ is either $0$ or at least $q+1$ for every $v_i \in A, v_j \in B$. 
\begin{enumerate}[label={\rm (N\arabic*)}]
	\item \label{N1} If $A=B$, then  $c\geq (q+1)|A|$.
	\item \label{N2} If $A\neq B$, then 
$c > (2-\frac{1}{2}\cdot \mathbbm{1}_{q=2})(|A|+|B|) + (2q-6) + 2J_{A,B} + (\frac{5}{2}-J_{A,B})\cdot \mathbbm{1}_{q=2}.	$	
\item \label{N3} If $A \neq B$ and $q=2$, then $c\geq |A|+2|B| + 1$.
\end{enumerate}
\end{lemma}
\begin{proof}
If $A = B$ then every pair of vertices in $A$ are at distance at least $q+1$ so $c\geq (q+1)|A|$, yielding \ref{N1}. 

Now suppose $A \neq B$. The vertices of $A \cup B$ partition $E(C)$ into segments. Each such segment has at least two edges, and if a segment connects a vertex in $A$ to a vertex in $B$, then it has at least $q+1$ edges. The total number of segments is $|A|+|B|-|A\cap B|$ and at least $|A\cap B| + J_{A,B}$ segments contains at least $q+1$ edges. Hence
\begin{align}\label{N33}
c &\geq (q+1)(|A\cap B|+J_{A,B})+2(|A|+|B|-2|A\cap B|-J_{A,B}) \\
\nonumber & = 2|A|+2|B| + (q-3)|A\cap B|+ (q-1) J_{A,B}.
 \end{align}

 If $J_{A,B}=1$, then $|A\cap B|\geq 1$ as $A,B$ are nonempty. Also $A\neq B$ implies $|A\cap B| \leq \frac{1}{2}(|A|+|B|-1)$, so we can verify that it is at least $(2-\frac{1}{2}\cdot \mathbbm{1}_{q=2})(|A|+|B|) + (2q-4) + \frac{3}{2}\cdot \mathbbm{1}_{q=2}$. On the other hand, if $J_{A,B}=2$, then it is at least $(2-\frac{1}{2}\cdot \mathbbm{1}_{q=2})(|A|+|B|) + (2q-2) + \frac{1}{2}\cdot \mathbbm{1}_{q=2}$. This provides \ref{N2}.
 
  If $A \neq B$ and $q=2$, then~\eqref{N33} together with $|A\cap B|\leq |A|$ implies \ref{N3}. \end{proof}

\begin{lemma}\label{pathbound}Let $Q = w_1, w_2, \ldots, w_q$ be a path and $A, B \subseteq V(Q)$ be two nonempty subsets such that each set does not contain any consecutive vertices in $Q$, and $w_{q-1} \notin B$. Suppose that for all $w_i \in A, w_j \in B$ with $i>j$, we have $i-j \geq t+1$. Then
\begin{enumerate}
	\item[(a)] if all $w_i \in A, w_j \in B$ satisfy $i \leq j$, then $q \geq 2(|A| + |B| - 1) -1$, and 
	\item[(b)] if $t\geq 2$ and there exists $i'>j'$ with $w_{i'}\in A$ and $w_{j'}\in B$,  then 
	$$q \geq 2|A| + |B| + t-4 + \mathbbm{1}_{w_q \notin A} + \mathbbm{1}_{A \cap B = \emptyset}.$$ 
	\end{enumerate}
\end{lemma}

\begin{proof}
Part (a) follows from the fact that $A \cap B$ contains at most one vertex. 
To prove part (b), fix a pair $i',j'$ as in the statement of the lemma with $i'-j'$ minimized.
Define
\[X = \{w_{k-1}, w_{k-2}: w_k \in A - \{w_1\}\} \cap V(Q), \qquad Y = \{v_{j'+1}, \ldots, v_{i'-3}\}.\] Since $i'-j' \geq t+1$, $|Y| \geq t-2$. Since $t\geq 2$, $X, Y,$ and $B$ are pairwise disjoint sets. Therefore $q \geq |X| + |B| + |Y| \geq 2(|A|-1) + |B| + t-2 = 2|A| + |B| + t-4$.

If $w_q \notin A$, then $w_{q-1} \notin X \cup B \cup Y$, so we further obtain $q \geq 2|A| + |B| + t-4 + 1$. 

Next suppose $A \cap B = \emptyset$. Let $\alpha$ be the largest index such that $w_\alpha \in A$. If $\alpha <q-1$, then $w_\alpha, w_{q-1}$ are two vertices not in $X \cup B \cup Y$ and we're done. If $\alpha \in \{q-1, q\}$, then $w_\alpha \notin X \cup B \cup Y$. We will find another vertex not in this set.
There exists a largest index $r < \alpha$ such that $w_r \in B$. If there exists $r' < r$ such that $w_{r'} \in A$, then the vertex with the largest such index is also not in $X \cup B \cup Y$, and we are done. If no such $r'$ exists, then $w_1 \notin A$, so $|X| \geq 2|A| - 1$. We instead get $q \geq 2|A| - 1 + |B| + t-2 + |\{w_{\alpha}\}|$.
\end{proof}

The following two extensions of Ore's theorem will be useful.
\begin{thm}[Linial \cite{L}]\label{thm: long cycle Ore}
For each $n\geq 3$, every $2$-connected $n$-vertex graph $G$ contains a cycle with length at least $\min\{ \sigma_2(G), n\}$.	
\end{thm}

\begin{thm}[Ore \cite{O2}]\label{thm: ham conn}
For each $n\geq 2$, every $n$-vertex graph $G$ with $\sigma_2(G) \geq n+1$ is hamiltonian-connected. That is, for every $u,v \in V(G)$, $G$ contains a hamiltonian $u,v$-path.
\end{thm}

\section{Structure of $G-C$ for an $L$-maximum cycle $C$}\label{seca}

\begin{definition}
For a given graph $G$ and a vertex set $X\subseteq V(G)$,
a cycle $C$ in $G$ is {\bf $X$-maximal} if the following holds.
\begin{enumerate}[label={\rm (L\arabic*)}]
\item \label{L1} $|C|$ is as large as possible,
\item \label{L2} subject to \ref{L1}, $|C\cap X|$ is as large as possible.
\end{enumerate}
\end{definition}

For a given graph $G$, let
$L = L(G)$ be the set of vertices in an $n$-vertex graph $G$ with degree less than $\sigma_2(G)/2$. 
We later utilize several properties of  $L$-maximum cycles.
In this section, we prove the following lemma.

\begin{lemma}\label{mainlem1} Let $n\geq 13$.
Suppose that $G$ is an $n$-vertex $2$-connected graph with $\sigma_2(G)\geq \frac{2n-3}{3}$. If $G$ does not contain a spanning jellyfish, then $G-C$ contains no edges for every $L(G)$-maximal cycle $C$.\end{lemma}

Suppose that $G$ is a $2$-connected $n$-vertex graph with $\sigma_2(G) \geq \frac{2n-3}{3}$ and no spanning jellyfish.  By Theorem~\ref{thm: long cycle Ore}, $|C|\geq \sigma_2(G)$.  We call a vertex $v$ {\bf low} if $d(v) \leq (n-2)/3$, {\bf normal} if $d(v) \geq (n-1)/3$, and {\bf high} if there exists a low vertex that is not adjacent to $v$. In particular, $d(v) \geq \lceil n/3 \rceil$ holds for every high vertex $v$. Then $L=L(G)$ denotes the set of low vertices in $G$. By the definition of $\sigma_2(G)$, all low vertices are adjacent to each other and $G[L]$ is a complete graph. 

Observe that if $u,v$ are non-adjacent vertices say with $d(u) \leq d(v)$, then either they are both normal, or $u$ is low and $v$ is high, in particular, $d(v)\geq \frac{n}{3}$. Thus in both cases, 
\begin{equation}\label{+2d}
d(u) + 2d(v) \geq (3n-3)/3 = n-1.
\end{equation}

We collect the following two definitions.
\begin{definition}
	Given an $L$-maximum cycle $C$, we call a pair $(x,y)$ of distinct vertices $C$-{\bf usable}, if $x,y\notin V(C)$,
each of $x$ and $y$ has a neighbor in $C$, and $|N_C(x)\cup N_C(y)|\geq 2$.
\end{definition}

\begin{definition}
	We say that a cycle $C$ is a {\bf semi-$2$-frame} if $G-C$ contains a single component $H$ with exactly $s:=\lceil \frac{n}{3} -2 \rceil$ vertices and $|N_C(H)|=2$.
	In addition, if $H=K_s$ and $H$ contains at least one low vertex, then it is a {\bf $2$-frame}. 
\end{definition}

Suppose $C = v_1, \ldots, v_c, v_1$. We write $v_{c+i}=v_i$ for each $i\geq 0$.  For a subset $X\subseteq V(C)$ and $k\in \mathbb{N}$, %we write $v^{+k}_i=v_{k+i}$ and $v^{-k}_{i}=v_{i-k}$ and 
\begin{align*}
X^{+k}= \{v_{i+k}: v_i\in X\} \text{ and } X^{-k}= \{v_{i-k}: v_i\in X\}.
\end{align*} When $X$ contains a single vertex $x$, we write $x^{+k}$ or $x^{-k}$ to mean $\{x\}^{+k}$ or $\{x\}^{-k}$, respectively. 
For brevity, we write superscripts $+$ and $-$ instead of $+1$ and $-1$, respectively.

To prove Lemma~\ref{mainlem1}, we assume there exists a component in $G-C$ with at least two vertices. From this assumption, we will deduce that $C$ must be a $2$-frame, and prove that it also contains a spanning jellyfish.

\subsection{All components in $G-C$ are complete.}

In order to reduce our case to $2$-frames, we first aim to prove that all components in $G-C$ are complete graphs. %We first make t
The following observation is helpful.
\begin{align}\label{eq: replace}
\begin{split}
	&\textit{For a path $P=u_1,\dots, u_s$ in $G-C$ and
distinct $v_i \in N_C(u_1) , v_j \in N_C(u_s)$,} 
\\ &\textit{the distance between $v_i$ and $v_j$ in $C$ is at least $s+1+I_{P\cap L}$. }	
\end{split}
\end{align}

Indeed, assuming $i<j$, if the segment between $v_i$ and $v_j$ contains no low vertices or $I_{P\cap L}=0$, then 
the cycle $v_1,\dots, v_i, u_1,\dots, u_s, v_j,\dots, v_s$ yields a longer cycle than $C$ or a cycle containing more low vertices, contradicting the choice of $C$ as an $L$-maximum cycle $C$.
Otherwise, choose a low vertex $v_{\ell}$ with $i<\ell < j$ and a low vertex $u_{\ell'}$ in $P$. As all low vertices are adjacent, either 
$v_1,\dots, v_i, u_1,\dots, u_{\ell'}, v_\ell, \dots, v_s$ or 
$v_1,\dots, v_{\ell}, u_{\ell'}, u_{\ell'+1}, \dots, u_{s},  v_j, \dots, v_s$
 yields a longer cycle, confirming \eqref{eq: replace}. 
 
 Using~\eqref{eq: replace}, we can obtain some properties of paths in $G-C$ as follows.

\begin{claim}\label{app1}
Let $P=u_1,\ldots,u_s$ be a path in $G-V(C)$. If $(u_1,u_s)$ is $C$-{usable},
then we have the following.
    \begin{enumerate}[label={\rm (A\arabic*)}]
\item \label{A1} $d_C(u_1)+d_C(u_s)\leq  \frac{2n+8}{3}-2s$.
\item \label{A2} $s\leq \frac{n-2}{3} -\frac{2}{3} \cdot  I_{P\cap L}$.
\item \label{A3} If $\min\{d_C(u_1),d_C(u_s)\}\geq \frac{n-1}{3}-(s-1)$, then $N_C(u_1)=N_C(u_s)$. Moreover, if  in addition $P\cap L\neq \emptyset$ holds, then $C$ is a semi-$2$-frame.\end{enumerate}
\end{claim}
\begin{proof} 
Let $A=N_C(u_1)$ and $B=N_C(u_s)$.
By the maximality of $C$ and  \eqref{eq: replace} with $s=1$, each of $A$ and $B$ consists of nonconsecutive vertices on $C$.
Since $(u_1,u_s)$ is $C$-{usable}, $A\cup B$ partition $C$ into segments and at least two of them has $s+1+I_{P\cap L}$ edges. Thus $n\geq 3s+2+2\cdot I_{P\cap L}$, hence $s\leq \frac{n-2}{3}-\frac{2}{3}\cdot I_{P\cap L}$, confirming \ref{A2}.
Moreover, $C$-usability implies that the cycle $C$ with the sets $A$ and $B$ satisfies the conditions of Lemma~\ref{neighbor2}.

If $A\neq B$, then Lemma~\ref{neighbor2} implies
\begin{align}\label{eq:AneqB}
	n-s\geq c \geq \left(2-\frac{1}{2}\cdot \mathbbm{1}_{s=2}\right)\left(|A|+|B|\right) +(2s-6) + 2J_{A,B}+ \mathbbm{1}_{s=2}\cdot(\frac{5}{2}-J_{A,B}).
\end{align}
If $s=2$, then~\eqref{eq:AneqB} becomes
$|A|+|B|\leq \frac{2}{3}n -2 \leq \frac{2}{3}n -2s+\frac{8}{3}.$
If $s\geq 3$, then we have 
$$|A|+|B|\leq \frac{n-3s+4}{2}\leq \frac{2}{3}n -2s+\frac{8}{3},$$
where the final inequality holds as $s\leq \frac{n-2}{3}$ by (A2). This confirms \ref{A1} when $A\neq B$.

Consider the case $A=B$. Then Lemma~\ref{neighbor2} yields $n-s\geq c\geq \frac{s+1}{2}(|A|+|B|)$, hence 
$$|A|+|B|\leq \frac{2(n-s)}{s+1} = \frac{2(n+s^2)}{s+1} -2s.$$
As  $2\leq s\leq \frac{n-2}{3}$ and the function $f(s)=\frac{2n+2s^2}{s+1}$ is convex, it  maximizes at either $s=2$ or $s=\frac{n-2}{3}$.  Since 
$f(2)= f(\frac{n-2}{3})=\frac{2n+8}{3}$, we get $f(s)\leq \frac{2n+8}{3} - 2s$ yielding \ref{A1}.

To prove \ref{A3}, assume $d_C(u_1), d_C(u_s)\geq \frac{n-1}{2} -(s-1)$ and $A\neq B$.
Because we assume $A\neq B$, if $|A|=|B|$ then $J_{A,B}=2$. Thus $J_{A,B}\leq 1$ implies $|A|+|B|\geq \frac{2n-2}{3}-2(s-1)+1$.
By this fact, if $(s,J_{A,B})=(2,1)$ or $(2,2)$, then  \eqref{eq:AneqB} yields
$n-2\geq n-1$ or $n-2\geq n-\frac{3}{2}$, respectively, a contradiction.
If $s\geq 3$, then for each case $J_{A,B}\in \{1,2\}$, \eqref{eq:AneqB} yields 
$n-s\geq \frac{4n-4}{3}-2s+2$, a contradiction to $s\leq \frac{n-2}{3}$.
 Thus $d_C(u_1), d_C(u_s)\geq \frac{n-1}{2} -(s-1)$ implies $A=B$, confirming the first statement of \ref{A3}.

If  in addition $P$ contains a low vertex, then Lemma~\ref{neighbor2}~\ref{N1} together with \eqref{eq: replace} implies that 
$n-s \geq \frac{s+2}{2}(|A|+|B|)$ instead.
Thus 
$$|A|\leq \frac{(n-s)}{s+2}= \frac{n+s^2+s}{s+2} -s.$$
Again,  $2\leq s\leq\frac{n-4}{3}$ in this case, and $f(s)=\frac{n+s^2+s}{s+2}$ is convex.
If $s$ is between $2$ and $\frac{n-5}{3}$, then this is maximized at $\frac{n-5}{3}$ as $f(2)< f(\frac{n-5}{3}) \leq \frac{n+1}{3}+\frac{3}{n+1}$, which is a contradiction because $2|A|=|A|+|B|\geq \frac{2n-2}{3} -2(s-1)> 2f(\frac{n-5}{3}) -2s$. Thus $s= \frac{n-4}{3}$ must hold with
$|A|\leq \frac{(n-s)}{s+2}\leq 2$.
As $G$ is $2$-connected, $A=B$ must have size exactly two. As $s=\frac{n-4}{3}\geq \lceil \frac{n}{3}-2\rceil$, this implies that the path together with $C$ covers all vertices in $G$, thus $C$ is a semi-$2$-frame.
This proves the second statement of \ref{A3}.
\end{proof}

We now prove the following claim regarding the length of maximal paths in $G-V(G)$.
\begin{claim}\label{cl: s}
For any component $H$ of $G-C$ with $|H|\geq 2$, 
there exists a longest path $P$ in $H$ whose pair of the end vertices of $P$ is $C$-usable.
Moreover, each path in $G-C$ has  at most $\frac{n-2}{3} -\frac{2}{3}\cdot I_{P\cap L}$ vertices.
\end{claim}
\begin{proof}
By Claim~\ref{app1}~\ref{A2}, the first statement implies the second statement.
To show the first statement, take a component $H$ of $G-C$ and assume for a contradiction that the 
pair of the end vertices of each
 longest path in $H$ is not $C$-usable. Suppose each longest path in $H$ has $s$ vertices.
Recall that $|C|\geq \frac{2n-3}{3}$, so $s\leq \frac{n}{3}+1$.

First, we claim that there exists a longest
path $P=u_1,\dots, u_s$ in $H$ with $d_C(u_1)>0$.
Indeed, if $u_1u_s\in E(G)$, then the $2$-connectedness of $G$ ensures a path from $P$ to $C$, yielding a desired path.
Otherwise, the maximality of $P$ ensures 
$2(s-2)\geq d(u_1)+d(u_s)\geq \sigma_2(G)\geq \frac{2n-3}{3}$, implying $s\geq \frac{n}{3}+\frac{3}{2}$, a contradiction to $s\leq \frac{n}{3}+1$. 
Thus there exists a longest path $P=u_1,\dots, u_s$ in which $u_1$ has a neighbor, say $v_c$, in $C$. Let us denote $u_0=v_c$.  By the maximality of $P$, 
%$u_1$ has no neighbors in the component of $G-C$ containing $P$, so 
 all neighbors of $u_s$ are in $V(C)\cup V(P)$.

Among those paths, choose a path $P=u_1,\dots, u_s$ so that if possible, $d(u_s)$ is normal.
If $u_s$ has a neighbor in $C-u_0$, then $(u_1,u_s)$ is a $C$-usable pair, so we are done. Assume that $u_s$ has no neighbors in $C-u_0$.
Let $u_{\ell_1},\dots, u_{\ell_t}$ be the neighbors of $u_s$ with $0\leq \ell_1\leq \dots \leq \ell_t$. Then 
\begin{equation}\label{dus}
\mbox{\em
$ s\geq d(u_s)+\ell_1$ with equality only when $N(u_s)=\{u_{\ell_1},u_{\ell_1+1},\ldots,u_{s-1}\}$.
}
\end{equation}
{\bf Case 1:} $u_s$ is normal.
Since $G$ is 2-connected, $G-v_c$ has a path $Q$ from some $v_i \in V(C)$ to some $u_j \in V(P)$  internally disjoint from $V(C) \cup V(P)$.
 Choose such $Q$ so that the index  $j$ is largest possible. 

 By symmetry, we may assume $i \geq \lceil c/2 \rceil$.  If $\ell_1 < j$, let $j'$ be the largest index smaller than $j$ such that $u_{j'} \in N(u_s)$.  Then the cycle \[C'=v_c, v_1, \ldots, v_i, Q, u_j, u_{j+1}, \ldots, u_s, u_{j'}, \ldots, u_1, v_c\] has at least $1 + \lceil c/2 \rceil + d(u_s)+\ell_1$ vertices. Since $|C'| \leq c\leq n-s$, by~\eqref{dus} we get
 $1+d(u_s)+\ell_1 \leq \frac{c}{2}\leq \frac{n-d(u_s)-\ell_1}{2}$ and hence $d(u_s)+\ell_1\leq \frac{n-2}{3}$. 
 By the case, this is not true, which means $j\leq \ell_1$, in particular $\ell_1\geq 1$.
 Then by~\eqref{dus}, $s\geq d(u_s)+\ell_1 \geq \frac{n+2}{3}+\ell_1-1$. Since $s\leq n-c\leq \frac{n}{3}+1$, this is possible only if $j=1$, $\ell_1=1$ and $N(u_s)=\{u_1,\ldots,u_{s-1}\}$. Moreover, by the maximality of $s$, $Q$ is just the edge $v_iu_1$. Since $G-u_1$ is connected, it  has a path $Q'$ from some $v_{i'} \in V(C)$ to some $u_{j'} \in V(P)-u_1$  internally disjoint from $V(C) \cup V(P)$. By the symmetry between $v_c$ and $v_i$, we may assume that $i'\neq c$, and then assume that $i' \geq \lceil c/2 \rceil$. 
 As $u_{j'-1}\in N(u_s)$, similarly to $C'$ the cycle \[C''=v_c, v_1, \ldots, v_{i'}, Q', u_{j'}, u_{j'+1}, \ldots, u_s, u_{j'-1}, \ldots, u_1, v_c\] has at least $1 + \lceil c/2 \rceil + d(u_s)+\ell_1$ vertices, and we again get that
 $d(u_s)+\ell\leq \frac{n-2}{3}$, a contradiction.
 
{\bf Case 2:} $u_s$ is low. By the choice of $P$, we can assume all such paths $P'=u_1', u_2', \ldots, u_s'$ with $u_1' = u_1$ end with a low vertex.
Suppose also that we chose $P = u_1, \ldots, u_s$ so that the smallest index $\ell_1$ of a
 neighbor of $u_s$ in $P$, $u_{\ell_1}$ is as small as possible. As the path $u_1\dots u_{\ell_1} u_s u_{s-1}\dots u_{\ell_1+1}$ is also a longest path with $u_1v_c\in E(G)$, the vertex $v_{\ell_1+1}$ is also a low vertex and has no neighbors in $V(C)\cup \{u_1,\ldots,u_{\ell_1-1}\}$; in particular, $u_{\ell_1+1}$ and $u_s$ are adjacent since they are both low. Repeating this argument, we get that all vertices $u_{\ell_1+1},u_{\ell_1+2},\ldots,u_s$ are low and have no neighbors in $V(C)\cup \{u_1,\ldots,u_{\ell_1-1}\}$. This contradicts the $2$-connectedness of $G$ and proves the claim.
\end{proof}

Let $u_1,\ldots,u_s$ be a longest path in $H$. %, then $d_C(u)$.  
 By Claim~\ref{cl: s}, $s\leq \frac{n-2}{3}$. So, 
\begin{equation}\label{sn}
\parbox{14cm}{\em if a normal vertex
$u$ is an end of a longest path $u_1,\ldots,u_s$ in $H$, then \\
$d_C(u)\geq d(u)-s+1\geq  \frac{n-1}{3}-\frac{n-2}{3}+1>1$.}
\end{equation}

Now we can prove the following claim as intended.
\begin{claim}\label{cl: complete}
Every component $H$ of $G-C$ is a complete graph.	
\end{claim}
\begin{proof}
Suppose that $H$ has $h$ vertices and $H\neq K_h$.
We first show that
\begin{equation}\label{eq: long path close}
	\textit{ every longest path $P=u_1,\dots, u_s$ in $H$ satisfies $u_1u_s\in E(G)$ and $s=h$.}
\end{equation}
If not, consider a longest path $P=u_1,\dots, u_s$ with $u_1u_s\notin E(G)$ with  $d(u_1)\leq d(u_s)$.
 Furthermore, among all such paths, choose $P$ so that $u_1$ has a neighbor $u_{\ell}$ in $P$ with maximum possible $\ell< s$. 

Since $d(u_1)\leq d(u_s)$ and $u_1u_s \notin E(G)$, $u_s$ is a normal vertex.
As $u_1u_s\notin E(G)$ and $P$ is a maximal path in $H$, each of $u_1,u_s$ has at most $s-2$ neighbors in $H$. Therefore 
\begin{equation}\label{nsn}
d_C(u_1)+ d_C(u_s)\geq \sigma_2(G)-2(s-2)\geq \frac{2n+9}{3}-2s,
\end{equation}
and Claim~\ref{app1} implies that $(u_1,u_s)$ is not $C$-usable. On the other hand, by~\eqref{sn}, $u_s$ has at least two neighbors in $C$.
Since $(u_1,u_s)$ is not $C$-usable, $d_C(u_1)=0$ and hence $u_1$ is a low vertex. Then $u_s$ is high.

Suppose $i\geq 2$, $u_1u_{i+1}\in E(G)$ and $u_i$ has a neighbor in $C$.
 Then the path $P_i=u_i u_{i-1}\dots u_{1}u_{i+1}\dots u_s$ has length $s$ and pair $(u_i,u_s)$ is $C$-usable. If $u_i$ has a non-neighbor in $P$, then~\eqref{nsn}  with $u_i$ in place of $u_1$ holds, which contradicts Claim~\ref{app1}. Otherwise, $d(u_i)\geq d(u_1)+1$, and instead of~\eqref{nsn} we have
 \begin{equation}\label{nsn5}
d_C(u_i)+ d_C(u_s)\geq 1+d(u_1)-(s-1)+d(u_s)-(s-2)\geq \sigma_2(G)-2s+4\geq \frac{2n+9}{3}-2s,
\end{equation}
again contradicting Claim~\ref{app1}. Thus for $2\leq i\leq \ell-1$,
\begin{equation}\label{nsn4}
\mbox{\em if   $u_1u_{i+1}\in E(G)$, then $u_i$ has no neighbor in $C$;  so it is low and $u_1u_i\in E(G)$.}
\end{equation}
By~\eqref{nsn4}, for each $1\leq i\leq \ell-1$, $u_i$ is low and has no neighbors in $C$. 

Since $u_\ell$ is not a cut vertex in $G$, there are $2\leq i\leq \ell-1$ and $\ell+1\leq j\leq s$ such that
$u_iu_j\in E(G)$. If $j<s$, then path $P_i$ would contradict the maximality of $\ell$; so $j=s$.
 Consider path $P_{s-1}=u_{s-1},u_{s-2},\ldots,u_{i+1},u_1,u_2,\ldots,u_i,u_s$. If $u_{s-1}$ is normal, then 
 by~\eqref{sn} it has a neighbor in $C$. Furthermore,
 since $u_s$ is high, $d(u_{s-1})+d(u_s)\geq\frac{2n}{3}$, and
 $d_C(u_{s-1})+d_C(u_s)\geq \frac{2n}{3}-(s-1)-(s-2)=\frac{2n+9}{3}-2s$, contradicting Claim~\ref{app1}.  Otherwise, 
 $u_{s-1}$ is low and hence $\ell=s-1$. Now, since $u_s$ is not a cut vertex in $G$, $u_{s-1}$ has a neighbor in $C$, and $(u_{s-1},u_s)$ is a $C$-usable pair. Then we get a contradiction to Claim~\ref{app1}
as in~\eqref{nsn5} with $u_{s-1}$ in place of $u_i$.
% The choice of $P$ ensures that no vertex in $\{u_1,\dots, u_{\ell-1}\}$ has a neighbor in $C$. Moreover, However, this means $G-v_{\ell}$ is disconnected, contradicting that $G$ is $2$-connected.
This proves the first part of \eqref{eq: long path close}. The second part follows from this and the connectivity of $H$.

Next we claim that $H$ is a regular graph. Indeed, take a spanning cycle $u_1,\dots, u_h,u_1$ from \eqref{eq: long path close} and take two vertices $u_i,u_{i+1}$ with different degrees. Without loss of generality, we obtain a path 
$P'=w_1,\dots, w_h= u_{i+1},u_{i+2},\dots ,u_{h},u_{1},\dots, u_i$ with $d_H(w_1)< d_H(w_h)$. 
As $d_H(w_1)< d_H(w_h)$, there exists $j$ satisfying $w_j\in N_H(w_h)\setminus \{j: w_{j+1}\in N_H(w_1)\}$. Then the new path $w_1,\dots, w_j, w_h, w_{h-1},\dots, w_{j+1}$ contradicts \eqref{eq: long path close}.

Equation \eqref{eq: long path close} implies that $H$ is complete if $h\leq 3$. Thus we may assume that is an $r$-regular graph for some $r$ and $h\geq 4$. We claim that $r> h/2$. 
To show this, take a spanning cycle $u_1,\dots, u_h,u_1$. If two consecutive vertices are normal, then we obtain a hamiltonian  path between two normal vertices $w,w'$ in $H$. Then both $w$ and $w'$ has at least $\frac{n-1}{3} - (h-1) \geq 2$ as $h=s\leq \frac{n-2}{3}$. Hence, the pair $(w,w')$ is $C$-usable and Claim~\ref{app1} yield that 
$$ 2(\frac{n-1}{3}- r)\leq  d_C(w)+d_C(w')\leq \frac{2n+8}{3} -2h,$$
implying $r\geq h-\frac{5}{3} > h/2$.
Therefore, $H$ is an $r$-regular graph with $r>h/2$, implying that $H$ is hamiltonian-connected by Theorem~\ref{thm: ham conn}. That is, there exists a hamiltonian  path between any non-adjacent vertices. This together with~\eqref{eq: long path close} implies that $H$ is a complete graph.
\end{proof}

%By this claim, every vertex outside $C$ is an end vertex of a hamiltonian path in its component of $G-C$, Claim~\ref{cl: s} implies the following.
%\begin{align}\label{obs: normal two nbrs}
%	\textit{ If $u$ is a normal vertex outside $C$, then $d_C(u)\geq \frac{n-1}{3} - (n-c-1) \geq 2$.}
%\end{align}

\subsection{$G-C$ has only one component}

In the pursuit of showing $C$ is a $2$-frame, we will now show that $G-C$ has only one component.
\begin{claim}
If $G-C$ contains a component with at least 2 vertices, then it does not contain any other component.
\end{claim}

\begin{proof}
%Assume there are at least two components $H$ and $H'$ in $G-C$ with $|V(H)| + |V(H')|\geq 3$.
%By Claim~\ref{cl: s} we can choose hamiltonian  paths
%$P_1 = u_1,\dots, u_s $ in $H$ { and } $P_2 = w_1,\dots, w_{s'}$ in $H'$
%so that the  pairs of their end vertices are $C$-usable unless $|P_i|=1$. Since the low vertices in $G$ form a clique, at most one of $H$ and $H'$ contains low vertices.

Suppose $G-C$ contains at least two components, and at least one of these components is not a singleton. Since the low vertices in $G$ form a clique, at most one  component of $G-C$ contains low vertices.

%$P_1 = u_1,\dots, u_s $ in $H$ { and } $P_2 = w_1,\dots, w_{s'}$ in $H'$
%so that the  pairs of their end vertices are $C$-usable unless $|P_i|=1$. Since the low vertices in $G$ form a clique, at most one of $H$ and $H'$ contains low vertices.

\begin{subclaim}\label{sub1}
	There exist components $H$ and $H'$ in $G - C$ with $|V(H)| + |V(H')| \geq 3$ that satisfy both of the following. 
	
	(a) $H$ and $H'$ have hamiltonian paths $P_1 = u_1, \ldots, u_s$ and  $P_2=w_1, \ldots, w_{s'}$ respectively with  $d(u_1) \geq d(u_s)$, $d(w_1) \geq d(w_s)$ such that the endpoints of each $P_i$ are $C$-usable unless $|V(P_i)| = 1$.
	
	(b) There exist two distinct vertices, one from each of $N_C(u_1)$ and $N_C(w_1)$ such that the distance on $C$ between them is at most two.
\end{subclaim}
\begin{proof}
Suppose that the subclaim is not true. Let $H$ and $H'$ be two components in $G-C$ with $|V(H)| + |V(H')| \geq 3$. By Claim~\ref{cl: s}, we can choose paths $P_1$ and $P_2$ satisfying (a). 
Let $A=N_C(u_1)$ and $B=N_C(w_1)$. Then by our assumption, any distinct vertices $v_i\in A$ and $v_{j}\in B$ satisfy $|i-j|\geq 3$. 
We consider two cases. As the roles of $P_1$ and $P_2$ are symmetric, these two cases cover all possibilities.

{\bf Case 1.} The two ends  $u_1, u_s$ of $P_1$ are distinct normal vertices.
In this case, Claim~\ref{app1}~\ref{A3} implies that $N_C(u_s)=A$. Then Lemma~\ref{neighbor2} together with \eqref{eq: replace} implies 
\begin{align}\label{eq: (s+1)A}
c\geq  (s+1)(\frac{n-1}{3}-s+1).
\end{align}

First suppose $s\geq 3$. Equation \eqref{eq: replace} and the assumption that (b) does not hold imply that the four sets $ A^{-}, A^{+2}, B, B^{+}$ are pairwise disjoint. Then we have 
$$n-s-s' \geq |A^{-}\cup A^{+2}\cup B\cup B^{+}| \geq 2|A|+ 2|B| \geq  2( \frac{2n-3}{3} - s-s'+2),$$
yielding $s+s' \geq \frac{n}{3}+2$.
However, this is a contradiction to $s+s'\leq n-|C| \leq \frac{n}{3} +1$.

Next suppose $ s=2$. Equation~\eqref{eq: (s+1)A} yields $c\geq n-4$, thus $G - C - H$ has  at most $2$ vertices.
In this case, the set $A$ partitions $E(C)$ into segments such that each segment has three edges, possibly except one segment having four edges. Indeed, otherwise we would have $c\geq 3(\frac{n-1}{3}-1)+2=n-2$, contradicting~\eqref{eq: (s+1)A}. If $w_1$ has a neighbor in $V(C) - A$, then (b) holds. 
%we can contract at least two edges in $C$ and still apply Lemma~\ref{neighbor2} to subtract two from the left hand side of \eqref{eq: (s+1)A} and derive a contradiction. 
Hence we may assume $B \subseteq A$. If $V(G)=V(C)\cup \{u_1,u_2,w_1\}$ then we have a spanning jellyfish with the cycle $C$ and a branch vertex in $B$. So suppose there is $x\in V(G)-V(C)- \{u_1,u_2,w_1\}$. Then $|A|\leq \frac{n-4}{3}$ and hence as $u_1, u_2$ are normal vertices, $d(u_1)=d(u_2)=\frac{n-1}{3}$. Now, since $w_1$ and $x$  are not adjacent to $u_1$,
it follows from $\sigma_2(G)\geq \frac{2n-3}{3}$ that they also are normal. To have $B \subseteq A$ and the vertices $w_1, x$  normal at the same time, we
need $w_1x\in E(G)$ and $A=B=N_C(x)$. So, we again have a spanning jellyfish.

% and $G-C- H$ contains another vertex $x$, otherwise we can find a spanning jellyfish. If $N_C(x) \subseteq A$ as well, then $N_C(x) \cap B = \emptyset$ and $d_C(x) + d_C(w_1) \leq |A| \leq (n-4)/3$. Without loss of generality, we may assume $d_C(x) \leq d_C(w_1)$ (since $x$ may play the role of $w_1$ if they are both in singleton components).
%Then \[d(u_1) + d(x) \leq \frac{n-4}{3} + (s-1) + \frac{n-4}{6} + 1 = \frac{2n-8 + n-4 + 12}{6} = \frac{n}{2} < \sigma_2(G).\]

%We must therefore have $N_C(x)$ contains a vertex in $v_i \in V(C) - A$, and $v_i$ has distance at most $2$ from a vertex in $A$. Again since $x$ may play the role of $w_1$ if they belong to different components, $x = w_2$. If $B= A$, then we can find a longer cycle than $C$ containing the component $H'$. Therefore $|B| \leq |A| - 1$, and
%\[d(u_1) + d(w_1) \leq |A| + 1 + (|A| - 1) + 1 = 2|A| + 1 \leq \frac{2(n-4)}{3} + 1 < \sigma_2(G).\]
%As consecutive vertices in $A$ are distance at most four apart
%
%
%the subclaim holds if $B\setminus A\neq \emptyset$.
%Thus assume $B\subseteq A$. Then any vertex in $B$ which is adjacent to $w_{s'}$ yields a jellyfish as $s'\leq 2$.
%Because $B$ is not empty, this implies $s'=2$ and $N_C(w_1)\cap N_C(w_2)=\emptyset$, thus
%$d_C(w_1)+d_C(w_2)\leq |A|\leq \frac{n-s-s'}{3}\leq \frac{n-4}{3}$.
%However, this yields 
%$$d(u_1)+ \min\{ d(w_1), d(w_2)\}\leq |A|+(s-1) + \frac{n-4}{6} \leq \frac{n-2}{2}< \frac{2n-3}{3}=\sigma_2(G),$$ 
%a contradiction.

%If $s=3$, then \eqref{eq: (s+1)A} yields $n-s'\geq \frac{4n-13}{3}$, a contradiction as $n\geq 13$. 

 {\bf Case 2.} $u_1$ is a low vertex and $s'=1$.
 In this case, $w_{1}$ is a high vertex.
Then we have $|A|<|B|$, thus 
Lemma~\ref{neighbor2}~\ref{N3} and~\eqref{+2d} imply
$$ n-s-1\geq c \geq |A|+|B|+|B|+1 
\geq (n-1) -(s-1) + 1 = n-s + 1,$$
a contradiction.
This proves the subclaim.
\end{proof}
%
%Similarly, Subclaim~\ref{sub1} also holds for $N_C(u_s)$ and $N_C(w_{s'})$. 
%Without loss of generality
Let $H$ and $H'$ be two components in $G-C$ satisfying Subclaim 1. By symmetry, we may assume that $H$ contains no low vertices. Let $P_1$ and $P_2$ be their associated hamiltonian paths of $H$ and $H'$ respectively given by Subclaim 1. 
 Suppose also that we chose $P_2$ so that if possible, $w_{s'}$ is low (recall $d(w_1) \geq d(w_{s'})$). 
Without loss of generality, let 
$v_1\in N_C(u_1)$ and $v_{c-\gamma}\in N_C(w_1)$ for some $\gamma \in \{0,1\}$ be two vertices with distance at most $2$ in $C$. Note that if $\gamma = 1$, then $v_c$ is not in the neighborhood of any vertex in $\{u_1, u_s, w_1, w_{s'}\}$.

If $v_i\in N(u_s)$ and $v_j\in N(w_{s'})$ with $j<i\leq j+s+s'-\gamma+I_{P_2\cap L}$, then choose such $i,j$ with minimum $i-j$. The cycle $ u_1, \dots, u_s, v_i, v_{i+1},\dots, v_{c-\gamma}, w_1,\dots, w_{s'}, v_{j}, v_{j-1},\dots, v_1 $
contains $c-\gamma +s+s' - (i-j-1)\geq c+1-I_{P_2\cap L}$ vertices.
If $I_{P_2\cap L}=0$, then it is a contradiction as $C$ is a longest cycle. If $I_{P_2\cap L}=1$, then it yields a contradiction to the $L$-maximality of $C$ as no vertices in $\{v_{j+1},\dots, v_{i-1}\}$ are low otherwise they would be adjacent to $w_{s'}$ (recall $w_{s'}$ was chosen to be low, if possible). Hence we have $i-j \geq s+s'+1-\gamma + I_{P_2\cap L}$.

{\bf Case 1.} All $v_i\in N(u_s)$ and $v_j\in N(w_{s'})$ satisfy $i\leq j$.
In this case, $N(u_s)$ and $N(w_{s'})$ share at most one vertex and each set does not contain any consecutive vertices.
Recall that either $s=1$ or $(u_1,u_s)$ is $C$-usable. By~\eqref{eq: replace}, $N_C(u_s)-\{v_1\}$ is nonempty and contained in $\{v_{s+2}, \ldots, v_{c-\gamma}\}$. By the case, $w_{s'}$ then has no neighbors in $\{v_1,\dots, v_{s+1}\}$.
Similarly, either $s'=1$ or $(w_1, w_{s'})$ is $C$-usable. We obtain $N_C(w_{s'})-\{v_{c-\gamma}\} \subseteq \{v_{1},\dots,v_{c-s'-\gamma-1}\}$. Therefore $u_s$ has no neighbors in $\{v_{c-s'-\gamma},\dots, v_{c-\gamma}\}$.

We apply Lemma~\ref{pathbound}(a) to $v_{s+2}, \ldots, v_{c-s'-\gamma-1}$ with sets $N_C(u_s)$ and $N_C(w_{s'})$ to obtain 

%Using $|N_C(u_s)\cap N_C(w_{s'})|\leq 1$, we count the number of vertices in $\{v_{s+2},\dots, v_{c+s'-\gamma-1}\}$, 
$$ c- \gamma - s-s'-2
\geq 2(\sigma_2(G)-(s-1)-(s'-1) - 1) -1 \geq 2(\sigma(G) - s - s')+1.$$
As $n-s-s'\geq c$, we have 
$$n-2s-2s'-\gamma-2 \geq \frac{4n-6}{3} - 2s - 2s' + 1,$$
a contradiction. 

{\bf Case 2.} There exists $i'>j'$ with $v_{i'}\in N(u_s)$ and $v_{j'}\in N(w_{s'})$. We apply Lemma~\ref{pathbound}(b) with $Q = v_1, \ldots, v_{c-\gamma}$, $A = N_C(u_s), B = N_C(w_{s'})$, and $t = s + s' - \gamma + I_{P_2 \cap L} \geq 2$. 

%If $H'$ contains low vertices, then all vertices in $H$ are high. Therefore $2 d(u_s) + d(w_{s'}) \geq \sigma_2(G) + d(u_s) \geq (2n-3)/3 + n/3 \geq (3n-3)/3$. Otherwise if all vertices in $H$ and $H'$ are normal, $2d(u_s) + d(w_{s'}) \geq 3(n-1)/3$. 

 Lemma~\ref{pathbound} together with~\eqref{+2d} implies 
\[n-s-s' - \gamma \geq c - \gamma \geq (n-1) - 2(s-1) -(s'-1) + s + s' - \gamma + I_{P_2 \cap L} - 4 + \mathbbm{1}_{w_{c-\gamma} \notin A} + \mathbbm{1}_{A \cap B = \emptyset},\]
which yields $n - s' \geq n-2 + I_{P_2 \cap L}+ \mathbbm{1}_{v_{c-\gamma} \notin A} + \mathbbm{1}_{A \cap B = \emptyset}$. So we must have
\begin{equation}\label{eq: Isets}I_{P_2 \cap L}+ \mathbbm{1}_{v_{c-\gamma} \notin A} + \mathbbm{1}_{A \cap B = \emptyset} \leq 1 \text{ with equality only if } s' =1 \text{ and } G - C = H \cup H'.
\end{equation}

We first consider the case that $I_{P_2 \cap L}+ \mathbbm{1}_{v_{c-\gamma} \notin A} + \mathbbm{1}_{A \cap B = \emptyset} =0$. Then both $H$ and $H'$ contain only normal vertices and $A \cap B \neq \emptyset$. By Claim~\ref{app1}~\ref{A3} applied to both $H$ and $H'$, $N_C(u) = A$ and $N_C(u') = B$ for all $u \in V(H), u' \in V(H')$. If $H$ and $H'$ are the only components of $G-C$, then $G$ contains a spanning jellyfish. Otherwise $n-s-s' -1 \geq c$, and  Lemma~\ref{pathbound} implies $n-s' - 1 \geq  n-2 + 0 + 0 + 0$. So $s' = 1$. But since $H$ and $H'$ both contain only normal vertices, we may reverse the roles of $A$ and $B$ and obtain also $s=1$, contradicting $s+s' \geq 3$. 

Now we consider the case $I_{P_2 \cap L}+ \mathbbm{1}_{v_{c-\gamma} \notin A} + \mathbbm{1}_{A \cap B = \emptyset} =1$. Then  $s' = 1$ and therefore $s \geq 2$. By~\eqref{eq: replace}, $v_{c-\gamma} \notin A$. We must have $A \cap B\neq \emptyset$, otherwise we violate~\eqref{eq: Isets}. Since we assume $H$ contains only normal vertices, Claim~\ref{app1}~\ref{A3} again implies $N_C(u) = A$ for all $u \in V(H)$. We obtain a spanning jellyfish using any vertex in $A \cap B$. 
\end{proof}

\subsection{$C$ is a $2$-frame}
Now we know that $G-C$ contains only one component $H$ which is a complete graph. 
In this subsection, we prove that $C$ must be a $2$-frame.

If $H$ contains no low vertices, then taking hamiltonian  paths between every pair of vertices in $H$ and applying Claim~\ref{app1}~\ref{A3} to them implies that all vertices in $H$ have the same neighborhood in $C$. This yields a spanning jellyfish.
Thus we assume that $H$ contains a low vertex.

If $H$ contains at least two normal vertices, then take a hamiltonian  path between them and apply Claim~\ref{app1}~\ref{A3}. Using Claim~\ref{cl: complete}, we conclude that $C$ is a $2$-frame. So we assume that $H$ has at most one normal vertex.
As $G-C$ is a complete graph $H$, the following claim is useful to analyze the structure of $G$.
%between $C$ and $H$.

\begin{claim}\label{cl: indep edge}
If the component $H$ with $|H|\geq 2$ in $G-C$ contains a low vertex and $G$ has two disjoint edges $v_iu$, $v_{i'}u'$ between $C$ and $H$, then $N(v_{i+1})\cap \{v_{i'+1},\dots, v_{i'+s+1}\}=\emptyset$.
\end{claim}
\begin{proof}
Let $u^*$ be a low vertex in $H$.
Assume $v_{i'+\alpha}\in N(v_{i+1})$ for some $1\leq \alpha \leq s+1$.
If $\{v_{i'+1},\dots, v_{i'+s+1}\}$ contains no low vertices, then the cycle
$$ v_{i+1}, v_{i'+\alpha}, v_{i'+\alpha+1}, \dots, v_{i}, u, P', u', v_{i'}, v_{i'-1},\dots, v_{i+1}$$
yields a cycle contradicting the $L$-maximality of $C$, where $P'$ is a hamiltonian  path in $H$ from $u$ to $u'$. If $v_{i'+\beta}$ is a low vertex with $1\leq \beta<\alpha$, then it is adjacent to $u^*$.
If $u^*\neq u$, then repeating the above argument with the edges $v_i u, v_{i'+\beta} u^*$ yields a cycle longer than $C$, a contradiction.
If $u^*= u$, then replacing the segment $v_{i'+1}\dots v_{i'+\beta-1}$ with a hamiltonian  path in $H$ from $u^*$ to $u'$ yields a cycle  longer than $C$, a contradiction.
\end{proof}

We introduce some definitions. 
\begin{definition}
	A vertex in $N_C(H)$ is a {\bf  private neighbor} of a vertex $u \in V(H)$ if its only neighbor in $H$ is $u$. The vertices in $N_C(H)$ partition $E(C)$ into segments. A {\bf simple segment} of $C$ is a segment $S= v_i, v_{i+1}, \ldots, v_j$ of $C$ such that $\{v_i, v_j\} \subseteq N(H)$ and no interior vertices of $S$ belong to $N(H)$. We say $S$ is a {\bf private} segment if there exists some $u \in V(H)$ such that $v_i$ and $v_j$ are private neighbors of $u$. Otherwise we say $S$ is a {\bf shared} segment. 
%	Given a cyclic orientation of $C$, for some $v_i \in N_C(H)$, we denote by $S_{i}$ the simple segment of $C$ that starts with $v_i$.
\end{definition}

 Every  segment has at least one interior vertex. Moreover, as $H$ is complete, \eqref{eq: replace} implies that every shared segment has at least $s+1$ interior vertices.

\begin{claim}
	If $H$ contains at most one normal vertex  and $s=|H|\geq 2$, then $C$ is a $2$-frame.
\end{claim}
\begin{proof}
As $G$ is $2$-connected, there exists a simple shared segment $S$, say $S=v_1,\dots, v_k$, and distinct $u,u' \in V(H)$ such that $v_1u, v_ku'\in E(G)$. Let us suppose we chose $S$ and $u, u' \in V(H)$  such that if possible, $uv_k \notin E(G)$.  Let $P = u_1, \ldots, u_s$ be a spanning path of $H$ with $u_1 = u$ and $u_s = u'$. 
By the maximality of $C$, $v_2, v_{k+1}$ have no neighbors in $H$ and therefore are high vertices.
Define
\begin{align*}
X_1 &= \{v_i: 3\leq i\leq k, v_i\in N(v_2)\}, &X_2 &= \{v_i: i>k, v_i\in N(v_2)\}, &  \\
Y_1&=N_C(u_1)-\{v_1,v_k\}, & Y_2&= N_C(u_s)-\{v_1,v_k\}  \enspace  \text{ and } 
& W =N(v_{k+1})\cup \{v_{k+1}\}.
\end{align*}
\begin{subclaim}\label{subcl1}
The sets $X_1^-, X_2^+, Y_1^{+2}, (Y_2-X_2)^+$ and $W$ are pairwise disjoint.
Moreover, $\bigcup_{t=1}^{s+1} X_2^{+t}\cap W=\emptyset$.
\end{subclaim}
\begin{proof}
By construction, $X_1^-\cap X_2^+=\emptyset$. Also $X_1^-\cap (Y_1^{+2}\cup Y_2^{+})=\emptyset$ as $S$ is a simple segment.
If $v_i\in X_1^{-}\cap W$, then the cycle 
$v_1,u_1,\dots, u_s, v_k, v_{k-1},\dots, v_{i+1}, v_2 ,v_3,\dots, v_{i}, v_{k+1}, v_{k+2},\dots, v_c, v_1$ is longer than $C$, a contradiction. Thus $X_1^{-}$ is disjoint from all other four sets.

If $v_i\in X_{2}^{+}\cap Y_1^{+2}$, then the cycle $v_{i-1}, v_{i}, \dots, v_c, v_1, u_1, v_{i-2}, v_{i-3},\dots, v_{2}, v_{i-1}$ is a cycle containing one more vertex than $C$, a contradiction.
%Similarly if $v_i\in X_{2}^+\cap Y_2^{+2}$, then applying Claim~\ref{cl: indep edge} to edges $v_1u_1, v_{i-2}u_{s}$ and $v_2 u_{i-1}$ yields a contradiction.
 Clearly $X_2^+\cap (Y_2-X_2)^{+}=\emptyset$. 
If $v_i\in X_2^{+t}\cap W$ for some $1 \leq t\leq s+1$, then $v_1, v_c, v_{c-1},\dots, v_{i}, v_{k+1}, v_{k+2},\dots, v_{i-t}, v_2, v_3,\dots, v_{k}, u_s, u_{s-1},\dots, u_1, v_1$ yields a better cycle than $C$ unless $t = s+1$, and $\{v_{i-t+1}, \ldots, v_{i-1}\}$ contains at least as many low vertices as $P$ does.
 In this case, let $1 \leq \beta\leq s$ be the largest index such that $v_{i-(s+1)+\beta}$ is low. Since $H$ contains at most one normal vertex, $u_1$ or $u_s$ is low, and moreover $\beta \in \{s, s-1\}$. If $u_1$ is low, then the cycle $v_{k+1}, v_{k+2}, \ldots, v_{i-(s+1) + \beta}, u_1, \ldots, u_s, v_k, v_{k-1}, \ldots, v_i, v_{k+1}$ is longer than $C$. If $u_s$ is low, we take the cycle $v_{k+1}, v_{k+2}, \ldots, v_{i-(s+1) + \beta}, u_s, v_k, v_{k-1}, \ldots, v_i, v_{k+1}$ which is longer than $C'$ unless $\beta = s-1$. By the choice of $\beta$,  $v_{i+\beta + 1} =  v_{i+s}$ is normal. Therefore this cycle has the same length as $C$ but one more low vertex $(u_s)$.
 This confirms the ``moreover" part of the statement. Thus $X_{2}^+$ is disjoint from the last three sets.

If $v_i\in Y_1^{+2}\cap Y_2^{+}$, then \eqref{eq: replace} yields a contradiction.
If $v_i\in Y_1^{+2}\cap W$, then applying Claim~\ref{cl: indep edge} with three edges $v_k u_s, v_{i-2} u_1$ and $v_{k+1} v_{i}$ yields a contradiction. Finally, if $v_i\in Y_2^+\cap W$, then 
$v_{k+1}, v_{k+2},\dots, v_{i-1}, u_s, v_{k}, v_{k-1},\dots, v_1, v_c, \dots, v_{i}, v_{k+1}$ yields a cycle longer than $C$, a contradiction. This proves the subclaim.
\end{proof}
%From this subclaim, we can make the following observation.
Now we prove that
\begin{align}\label{eq: nonneighbor obs}
\textit{every vertex in $H$ has a neighbor in $C$.}	
\end{align}
Indeed, as $H$ is a clique, if \eqref{eq: nonneighbor obs} is not true, then there is a vertex  $w$ with degree exactly $s-1 \leq \frac{n-7}{3}$ by Claim~\ref{cl: s}. Then each of $v_2$ and $v_{k+1}$ has degree at least $\sigma_2(G)-\frac{n-7}{3}\geq \frac{n+4}{3}$.
As $u_1v_1\in E(G)$, $d(u_1)\geq s$, thus $d(u_1)+ d(v_2)\geq d(w)+ d(v_2)+1 \geq \frac{2n-3}{3}+1$.
By Subclaim~\ref{subcl1}, we have
\begin{align*}
\begin{split}
n-s &\geq |X_1|+|X_2|+|Y_1|+|Y_2\setminus X_2| + |W| \geq (d(v_2)-1) + (d(u_1)-2-(s-1)) + 0 + d(v_{k+1})+ 1 \\
& \geq  \frac{2n-3}{3}+1-(s-1) + \frac{n+4}{3}-2 \geq n - s +\frac{1}{3}.
\end{split}
\end{align*}
 This contradiction proves \eqref{eq: nonneighbor obs}.

On the other hand, as $G$ does not contain a spanning jellyfish, every vertex in $C$ has a non-neighbor in $H$. This together with \eqref{eq: nonneighbor obs} implies that by the choice of $(S, u, u')$, $v_k \notin N(u_1)$. Then $|Y_1| = d(u_1)-(s-1)-1$, and we instead obtain
\begin{align}\label{eq: sum size 2}
n-s \geq |X_1|+|X_2|+|Y_1|+|Y_2 - X_2| + |W| \geq d(v_2)+ d(u_1)+ d(v_{k+1})+ |Y_2-X_2| -s.
\end{align}

Our goal is to find another vertex set disjoint from $X_1^-,X_2^+, Y_1^{+2}, Y_2^{+}$ and $W$ to improve \eqref{eq: sum size 2} in several situations.

By the $2$-connectedness of $G$, $|N_C(H)|\geq 2$. 
First assume that $N_C(H)=\{v_1,v_k\}$.
So $d(u_1) \leq 1 + (s-1) = s$ and therefore $d(v_{k+1}) \geq \sigma_2(G)-s$. Equation
\eqref{eq: sum size 2} yields
$n-s \geq \sigma_2(G) + (\sigma_2(G)-s) +0-s \geq \frac{4n-6}{3} -2s$, implying $s\geq \lceil \frac{n}{3}-2\rceil$. 
Using Claim~\ref{cl: complete}, $C$ is a $2$-frame.

Hence assume that $|N_C(H)|\geq 3$. Then there are at least three simple segments, so choose $S'=v_{i}\dots v_{i+t}$ with $i\notin \{1,k\}$.
We now verify the following.
\begin{align}\label{eq: assumption}
\begin{split}
&\hbox{The vertex $v_i$ is a private neighbor of $u_1$ or $u_s$.}
\end{split}
\end{align}
Indeed, if not, then let $u_{j'}\neq u_1$ and $u_{j''}\neq u_s$ be neighbors of $v_i$ (possibly $u_{j'} = u_{j''}$).
Then we claim that $Z_1=\{v_{i+3},\dots v_{i+s+1}\}$ does not intersect with $X_1^{-}, X_2^+, Y_1^{+2}, (Y_2-X_2)^+$ or $W$.
It is trivial that it is disjoint from $X_1^{-}$.
If it intersects $X_2^+$, then Claim~\ref{cl: indep edge} with two edges $u_1v_1, u_{j'} v_i$ yields a contradiction.

If it intersects $Y_1^+$ or $(Y_2-X_2)^+$, 
applying \eqref{eq: replace} to a sub-segment of $v_{i}\dots v_{i+s}$ with a hamiltonian  path in $P$ either between $u_1$ and $u_{j'}$ or between $u_s$ and $u_{j''}$ yields a contradiction.
If it intersects $W$, then Claim~\ref{cl: indep edge} with two edges $u_sv_k, u_{j''} v_i$ yields a contradiction.
With this additional segment disjoint from those sets, we can add $|Z_1|=s-1$ to the right hand side of \eqref{eq: sum size 2}, thus
$$n-s \geq d(v_2)+ d(u_1) + d(v_{k+1})-s + (s-1) \geq \frac{2n-3}{3} + \frac{n}{3} -1 \geq n-2.$$
This yields $s=2$ and the assumption that $v_i$ is not a private neighbor of $u_1$ or $u_s$ implies that $v_i$ is adjacent to both $u_1$ and $u_s$, yielding a jellyfish, a contradiction. Thus this confirms \eqref{eq: assumption}.

Applying \eqref{eq: assumption} to all segments, we obtain %the following.  
 \begin{equation}\label{int}\hbox{$N_C(w) \subseteq\{v_1,v_k\}$ for all $w \in V(H) - \{u_1, u_s\}$, and $N_C(u_1) \cap N_C(u_s) \subseteq \{v_1\}$.} \end{equation}

Furthermore, we claim that $X_2\cap Y_2=\emptyset$. Otherwise, there exists $v_j \in N(v_2)\cap N(u_s)\setminus\{v_1,v_k\}$.
Consider $Z_2=\{v_{j+2},\dots, v_{j+s+1}\}$. By definition, it is disjoint from $X_1^{-}$.  As $v_j\in X_2$, Subclaim~\ref{subcl1} implies that $Z_2$ is disjoint from $W$.
As $v_j\in Y_2$, if $Z_2\cap Y_1^{+2}\neq \emptyset$, then applying \eqref{eq: replace} to a simple shared segment of length at most $s-1$ yields a contradiction. 
If $Z_2\cap X_2^{+}\neq \emptyset$, then Claim~\ref{cl: indep edge} with the two edges $v_1u_1, v_j u_s$ yields a contradiction.
Thus we have
\begin{align*}
n-s &\geq |X_1^{-}|+|X_2^+|+|Y_1^{+2}| + |W|+ |Z_2| \\
&\geq 
(d(v_2)-1)+(d(u_1)-s)+d(v_{k+1})+1 + s
\geq \frac{2n-3}{3}+ \frac{n}{3} \geq n-1,	
\end{align*}
so $s=1$, a contradiction. We conclude that $Y_2 = Y_2 -X_2$.
% so $s=2$. 
% However if $s=2$, then consider $v_{j+3}$ in addition. It does not belong to $X_1^{-}$ by definition, and does not belong to $X_2^{+}$, as otherwise Claim~\ref{cl: indep edge} with three edges $u_1v_1, u_sv_j, v_2 v_{j+2}$ again yields a contradiction. If it belongs to $Y_1^{+2}$, then applying \eqref{eq: replace} to the segment $v_{j}v_{j+1}$ with a hamiltonian  path in $H$ yields a contradiction. If it belongs to $W$, then the cycle $C'= v_{k+1}v_{j+3} v_{j+4}\dots v_c v_1 u_1 \dots u_s v_{k} v_{k-1}\dots v_2 v_{j} v_{j-1}\dots v_{k+1}$ yields a cycle with the same length with $C$. If $\{v_{j+1},v_{j+2}\}$ does not contain a low vertex, then this contains more low vertices than $C$ as $H$ contains a low vertex. 
% If one of $v_{j+1},v_{j+2}$ is a low vertex, then it is adjacent to a low vertex in $H$. Whether $u_1$ or $u_2$ is low, we can alway apply \eqref{eq: replace} to obtain a contradiction. Thus the left hand side of the above display inequality can be replaced with $|C\setminus\{v_{j+3}\}| = n-s-1=n-3$, thus we obtain $n-3\geq n-2$, a contradiction. 
%From this contradiction, $X_2\cap Y_2=\emptyset$ as desired and we
%conclude $Y_2=Y_2\setminus X_2$.

With this, \eqref{eq: sum size 2} becomes 
$n-s \geq n-1+|Y_2|-s$, so   $|Y_2|\leq 1$.

Furthermore, we claim that either $v_1\notin N(u_s)$ or $N_C(u_1)=\{v_1\}$ holds.
Indeed, suppose $u_1$ has a neighbor in $C$ other than $v_1$. Then there must be another shared segment $S' = v_a,\dots,v_b$ such that $v_a\in N(u_s)$ and $v_b\in N(u_1) - N(u_s)$ with $v_1\notin \{v_a,v_b\}$.
Equation \eqref{int} similarly applies to $S'$ with the roles of $u_1$ and $u_s$ reversed. Therefore $N_C(u_1)\cap N(u_s) \subseteq \{v_a\}\neq \{v_1\}$, so $v_1\notin N(u_s)$ as desired. 

So we now have two remaining cases.
If $N_C(u_1)=\{v_1\}$, as we assumed $|N_C(H)|\geq 3$, $|Y_2|\geq 1$ must hold.
Then $d(v_{k+1})\geq \sigma_2(G) - d(u_1) \geq \frac{2n-3}{3}-s$, and
\eqref{eq: sum size 2} becomes
$$n-s \geq \frac{2n-3}{3} + \frac{2n-3}{3} -s + |Y_2|-s \geq \frac{4n}{3} -2s-1,$$
implying $s\geq \frac{n}{3}-1$, a contradiction to Claim~\ref{app1}~\ref{A2}.

 If $v_1\notin N(u_s)$, then as 
 $d(v_2), d(v_{k+1}) \geq \frac{2n-3}{3} -\delta(G) \geq \frac{2n-3}{3} - d(u_s) \geq \frac{2n-3}{3}-(1+|Y_2|+(s-1))$, 
  \eqref{eq: sum size 2} becomes
 $$n-s \geq  \left\lceil \frac{2n-3}{3} + \frac{2n-3}{3} - |Y_2| -s  + |Y_2| -s \right \rceil,$$
 yielding $s\geq \lceil \frac{n-6}{3}\rceil$ 
with equality only when $d(u_1) = d(u_s) = \delta(G)$ and $\delta(G)=s+|Y_2|$.
As a strict inequality yields $s\geq \frac{n-3}{3}$, we must have equality. If $|Y_2| = 0$, then $d_C(u_1) = d_C(u_s) = 1$ and therefore $N_C(u_1) = \{v_1\}$, a contradiction. So suppose $|Y_2| \geq 1$. 
Because $s\geq \frac{n}{3}-2>2$ as $n\geq 13$, there exists $w\in V(H)\setminus\{u_1,u_s\}$ with $d(w)\geq \delta(G) = s+1$.
This together with \eqref{int} implies that $N_C(w)=\{v_1,v_k\}$. However, we can consider a new hamiltonian  path $P'$ from $u_1$ to $w$. We apply~\eqref{int} to the segment $S$ with $P'$, letting $w$ play the role of $u_s$. Then $N(u_s)\subseteq \{v_1,v_k\}$, a contradiction that $|Y_2|\geq 1$. This contradiction proves the claim.
\end{proof}

\subsection{Proof of Lemma~\ref{mainlem1}}
By the final claim in the above section, now we know that if $G-C$ contains an edge, then $C$ is a $2$-frame.
We finish the proof of Lemma~\ref{mainlem1} by showing every $2$-frame has a spanning jellyfish.

\begin{proof}[Proof of Lemma~\ref{mainlem1}]

Assume that $G-C$ contains at least one edge.
By previous claims, $C$ is a $2$-frame.
By Claim~\ref{cl: s} $H=G-C$ has a hamiltonian  path $P=u_1,\dots, u_s$  such that $(u_1,u_s)$ is $C$-usable.
Let $N_C(H)=\{v_1, v_k\}$. Let $S_1=v_1, \ldots, v_k, S_2=v_k, v_{k+1}, \ldots, v_1$ be two shared segments, and let $I_1, I_2$ the sets of their interior vertices, respectively. Then \eqref{eq: replace} implies that  each $S_i$ contains at least $s+1$ interior vertices. By the definition of $2$-frame, we have $(n-6)/3\leq s\leq (n-4)/3$.
 Thus assuming $|I_2|\geq |I_1|$, by~\eqref{eq: replace} we have 
$$(|I_1|,|I_2|)= \left\{\begin{array}{ll}
(s+1,s+1) & \text{ if } s=(n-4)/3, \\
(s+1,s+2) & \text{ if } s=(n-5)/3, \\
(s+1,s+3) \text{ or } (s+2,s+2) & \text{ if } s=(n-6)/3.
 \end{array}\right.
 $$
 As $C$ is a $2$-frame, $H$ contains a low vertex and all vertices in $V(C)-\{v_1,v_k\}$ are high vertices. Moreover, in each of $I_1$, $I_2$, and $P$ there exists a vertex with a non-neighbor in $\{v_1,v_k\}$, otherwise we could find a spanning jellyfish. In particular, $\delta(G) \leq s$, and each high vertex has degree at least $\lceil \sigma_2(G) - s \rceil$ which is at least $s+2, s+3, s+3$ when $s = (n-4)/3, (n-5)/3, (n-6)/3$, respectively. 
We claim that
\begin{align}\label{eq: int c}
\textit{$N_{I_1}(v_c)\subseteq \{v_2\}$. If $|I_1|=s+1$, then $N_{I_1}(v_c)=\emptyset$}	
\end{align}
Indeed, if $v_c$ is adjacent to some $v_i\in I_1$, then 
$v_c,v_i, v_{i-1},\dots, v_1, u_1, u_2,\dots, u_s, v_k, v_{k+1,}\dots, v_c$
has length $|C|-|I_1|+(i-1)+s\geq |C|$.
If $i\neq 2$ or $|I_1|=s+1$, then this yields either a longer cycle or cycle with more low vertices than $C$, a contradiction to the $L$-maximality of $C$, confirming \eqref{eq: int c}.

Let $v_i \in I_1$ be a vertex with a non-neighbor in $\{v_1, v_k\}$. Then $d_{S_1}(v_i) \leq |S_1| -1 -1 \leq |I_1|$ which is at most $d(v_i) -2$ if $(|I_1|, |I_2|) = (s+1, s+3)$, and at most $d(v_i) - 1$ otherwise. So there exists $v_j \in N_{I_2}(v_i)$. 

Suppose $|I_1| = s+1$. If $v_cv_{j-1} \in E(G)$, then $C' =v_1, \ldots, v_i, v_{i+1}, \ldots, v_c, v_{i-1}, \ldots, v_k, u_s, \ldots, u_1, v_1$ contradicts the $L$-maximality of $C$. Thus $d(v_c) = d_{S_2}(v_c) \leq |V(S_2) - \{v_c\}| - 1 = |I_2|$ which is less than $d(v_c)$ unless $|I_2| = s+3$. But in this case, $v_i$ must have another neighbor $v_{j'} \in I_2$ for which $v_{j'-1} \notin N(v_c)$. Instead we get $d(v_c) \leq |I_2|-1 < d(v_c)$. 

Finally we consider the case $(|I_1|, |I_2|) = (s+2, s+2)$. Symmetrically to~\eqref{eq: int c}, $N_{I_1}(v_{k+1}) \subseteq \{v_{k-1}\}$, and since $|I_1| = |I_2|$, $N_{I_2}(v_2) \subseteq \{v_c\}$. If both $v_cv_2$ and $v_{k-1}v_{k+1}$ are edges, then $v_1, v_2, v_c, v_{c-1}, \ldots, v_{k+1}, v_{k-1}, v_k, u_s, \ldots, u_1, v_1$ contradicts the $L$-maximality of $C$. Without loss of generality, suppose $v_cv_2 \notin E(G)$. Then $d(v_c) \leq d_{S_2}(v_c) \leq |V(S_2)|  -1 = s+3$ which implies $N(v_c) = V(S_2) - \{v_c\}$. Similarly,
% for $v_2$ and $S_1$. In particular, 
 $v_{j-1} \in N(v_c)$ and $v_i \neq v_2$. The cycle $C'$ from above again contradicts the $L$-maximality of $C$. 
\end{proof}

\section{Proof of Theorem~\ref{main}}\label{secb}
In this section, assume that $G$ is a $2$-connected graph with $\sigma_2(G) \geq (2n-3)/3$ containing no spanning jellyfish. By Lemma~\ref{mainlem1}, $G-C$ contains only isolated vertices for every $L$-maximum cycle $C$ of $G$. 
We will need a new version of the Hopping Lemma. This lemma was proved by Woodall~\cite{W} to attack problems on hamiltonian cycles. Jackson~\cite{BJ} refined it to prove that every 2-connected, $k$-regular graph on at most $3k$ vertices is hamiltonian. Another modification was proved by van den Heuvel~\cite{H}
and used in~\cite{CFHJL} to prove Theorem~\ref{Chen2}.

\subsection{Modified Hopping Lemma}

In this subsection, we develop a version of the Hopping Lemma in a more general form than we need in this paper, because we think that this form may be of interest by itself.
  In the next section we will apply it to our setting to prove Theorem~\ref{main}.

Assume $C$ is a longest cycle in a $2$-connected graph $G$ such that $G-C$ contains only singleton components. % (for instance, any $L$-maximal cycle satisfies this by the Lemma~\ref{mainlem1}). 
Then for all $u \in V(G-C)$, we have $N(u) = N_C(u)$, and $u$ does not have consecutive neighbors in $C$. 
Let $C = v_1, v_2, \ldots, v_c, v_1$.

\bigskip
 We iteratively define the sets $Y_0 (=Y_0(C)), Y_1, Y_2, \ldots$ and $X_1, X_2, \ldots$ as follows.

\begin{enumerate}[label={\rm (X\arabic*)}]
\item \label{XY1} $Y_0 = V(G-C)$,
\item\label{XY2} For $i \geq 1,$ $X_i = N_C(Y_{i-1})$, and
\item\label{XY3} For $i \geq 1$, $Y_i =\{v_i \in V(C): v_{i-1}, v_{i+1} \in X_i\}\cup Y_{i-1}$.  
\end{enumerate}
Then $Y_0 \subseteq Y_1 \subseteq Y_2 \subseteq \ldots$ and $X_1 \subseteq X_2 \subseteq \ldots$. As $n$ is finite, we may define $X (= X(C)) = \lim_{i \to \infty} X_i$ and $Y (=Y(C)) = \lim_{i\to \infty} Y_i$. For $W\in \{X,Y\}$, if $u\in W$, then
 the {\bf $W$-height} of $u$ is defined as $h_W(u) = \min_{i \in \mathbb N} \{i: u \in X_{i}\}$.
If $u \notin W$, then we set $h_W(u) = \infty$.

The {\bf height} of an $x,x'$-path $P$ is defined to be $h(P) =\max\{h_X(x),h_X(x')\}$.

\begin{definition}
A path $P=w_1\dots w_c$ is a {\bf $C$-hopping path} if the following hold.
\begin{enumerate}[label={\rm (H\arabic*)}]
\item \label{HP1} $w_1, w_c \in X$, 
\item \label{HP2}$P$ does not contain any consecutive vertices in $X_1$, 
\item \label{HP3}$V(P) =V(C)$, and
\item \label{HP4}if $j < h(P)$ and $w_s \in Y_j - \{w_1,w_c\}$, then $w_{s-1}, w_{s+1} \in X_j$. 
\end{enumerate}
\end{definition}
As $C$ is clear from the context, we often omit $C$ and just write a hopping path.

\medskip

\begin{claim}\label{hopreduction}
If $G$ contains a $C$-hopping path and no vertices in $X_1$ are consecutive in $C$, then it contains a $C$-hopping path of height 1. 
\end{claim}
\begin{proof}
Among all  hopping paths, choose $P = w_1, \ldots, w_c$ with $x = w_1, x' = w_c$ such that $h(P)$ is minimized, and subject to this $h_X(x) +h_X(x')$ is minimized. Without loss of generality, suppose $1 \leq h_X(x) \leq h_X(x')$ and let $i= h_X(x')>1$.

If $x\in Y$ and $h_Y(x) =j < i$, say $x = v_s$ in $C$, then by the definition of $Y_j$, $v_{s-1} \in X_j$, and by the assumption that $C$ does not contain two consecutive vertices in $X_1$, 
 the path $P' = v_s, v_{s+1}, \ldots, v_c, v_1, \ldots, v_{s-1}$ is a hopping path such that either $h(P') < h(P)$ or $h(P') = h(P)$ and $h_X(v_s) + h_X(v_{s-1}) \leq h_X(v_s) + j < h_X(x) + h_X(x')$, contradicting the choice of $P$. Similarly if some $v_s \in V(C)$ belongs to $Y_j \cap X_j$ with $j < i$, then the path $P'$ above is a hopping path of height at most $j < h(P)$. Therefore
\begin{equation}\label{hopeq}
\hbox{$h_Y(x) \geq i$ ~ and ~for all $j < i$ and $s \in [c]$, we have $w_s \notin X_j \cap Y_j$.}
\end{equation}

{\bf Case 1}: $h_X(x) < h_X(x')$.
Say $x' \in X_i- X_{i-1}$. By definition, $x'$ has a neighbor $w_j \in Y_{i-1} - Y_{i-2}$, and $w_j \neq x$ by~\eqref{hopeq}. Then $w_{j+1} \in X_{i-1}$ by (H4). Consider the path $P' = w_{j+1} w_{j+2} \dots w_c w_j w_{j-1} \ldots w_1$ which has height $h(P') = i-1$. \ref{HP1} and \ref{HP3} hold for $P'$. To check \ref{HP2}, it suffices to check it for the vertices $x'$ and $w_j$ since the $P'$-neighbors for all other vertices are the same as their $P$-neighbors. By assumption $h_X(x')> 1$, so $x' \notin X_1$. Therefore \ref{HP2} holds. Similarly, for \ref{HP4} it suffices to check $x'$ and $w_j$. In fact, since $h_Y(w_j) = i-1 = h(P')$, we need only check $x'$. If $x' \in Y_q$ and $q < i-1$, then $w_j \in X_{q+1} \cap Y_{i-1} \subseteq X_{i-1} \cap Y_{i-1}$, contradicting~\eqref{hopeq}.

{\bf Case 2}: $h_X(x) = h_X(x')=i >1$. There exists $w_j \in Y_{i-1}$ and $w_{j'} \in Y_{i-1}$ such that $h_Y(w_j) = h_Y(w_{j'}) = i-1$,  $xw_j \in E(G)$, and $x'w_{j'} \in E(G)$. By~\eqref{hopeq}, $w_j \neq x'$ and $w_{j'}\neq x$. 
 If $j \leq j'$ then set $P' = w_{j-1} w_{j-2} \ldots w_1, w_j \ldots w_{j'}, w_c w_{c-1} \dots w_{j'+1}$. Then $h(P') \leq \max\{h_X(w_{j-1}), h_X(w_{j'+1})\} \leq i-1$, and \ref{HP1}--\ref{HP3} hold. To verify \ref{HP4}, it suffices to check for $x, w_j, x', w_{j'}$. Symmetrically, we will just check $x$ and $w_j$. Indeed, $h_Y(w_j) \geq h(P')$ and by~\eqref{hopeq}, $h_Y(x) \geq i = h(P) > h(P')$, so \ref{HP4} holds. 
If $j > j'$, we instead take $P'' = w_{j'+1} \dots w_j w_1 w_1 \dots w_{j'} w_c w_{c-1} \ldots w_{j+1}$. A similar argument shows that $P''$ is a hopping path with $h(P'') \leq i-1$. This proves the claim. 
\end{proof}

We say a path $P =x_1, x_2, \ldots, x_p$ is a {\bf good path} if $N(x_1), N(x_p) \subseteq V(P)$ and each of $N(x_1)$ and $N(x_p)$ does not contain consecutive vertices in $P$.

\begin{claim}\label{cl: no hopping}
Suppose a longest good path in $G$ contains at most $c+1$ vertices. Then $C$ does not contain consecutive vertices in $X_1$, and $G$ has no $C$-hopping path,
%
%
%If $G$ does not contain a good path with $c+2$ vertices, then \textcolor{red}{no two consecutive vertices in $C$ both belong to $X_1$ and }
%$G$ contains no $C$-hopping path.
\end{claim}
\begin{proof}
 Assume that $G$ does not contain a good path with $c+2$ vertices. If $C$ contains two consecutive vertices $v_{c}, v_{c+1}$ in $X_1$, then there exists $y,y'$ in $Y_0$ that are neighbors of $v_{c}, v_{c+1}$, respectively. If $y=y'$, then $G$ contains a cycle of length $c+1$, a contradiction. If $y\neq y$, then $P=y v_{c} v_{c-1},\dots, v_{c+1} y'$ is a good path with $c+2$ vertices, as each of $y,y'$ does not contain consecutive neighbors in $C$ and therefore in $P$. Hence, we conclude that $C$ does not contain any consecutive vertices in $X_1$.
 
If $G$ has a hopping path, then by Claim~\ref{hopreduction} and the above conclusion we may assume $G$ contains a hopping path $P=w_1\dots w_c$ of height $1$. Let $x, x'$ be the endpoints of $P$.  By the definition of $X_1$, there exists $y, y'$ in $Y_0 = V(G-C)$ that are neighbors of $x$ and $x'$ respectively. If $y = y'$, then $G$ contains a cycle of length $c+1$, contradicting the choice of $C$. If $y \neq y'$, then by (H2), $y$ and $y'$ do not have any consecutive neighbors in $P$. Thus $y,P,y'$ is a good path with $c+2$ vertices.
\end{proof}

\begin{lemma}[Modified Hopping Lemma]\label{hopping}
Suppose $G$ is a graph with no good path with more than $c+1$ vertices.
Then the following hold where $X=X(C)$ and $Y=Y(C)$ as defined above. 
\begin{enumerate}[label={\rm (M\arabic*)}]
\item \label{MH1} $X$ does not contain consecutive vertices in $C$,
\item\label{MH2} $X\cap Y=\emptyset$ and $N(Y) \subseteq X$, and
\item\label{MH3} $Y$ is an independent set.
\end{enumerate}
\end{lemma}
\begin{proof}
By Claim~\ref{cl: no hopping}, $G$ contains no $C$-hopping path and $C$ does not contain two consecutive vertices in $X_1$. First observe that if $C$ contains some consecutive vertices $v_i$ and $v_{i+1}$ in $X$, then the path of length $c-1$ between them yields a hopping path, a contradiction. This proves \ref{MH1}.

%By definition, $N_C(Y) = X$. 
If $v_i\in Y_j$, then $v_{i+1}\in X_{j}$ and \ref{MH1} implies $v_i\notin X$, so $X\cap Y=\emptyset$.
On the other hand, by definition $N_C(Y)=X$. If $v_i\in Y_j$ has an additional neighbor $w$ outside $C$, then $v_i\in N(Y_0) \subseteq X$, a contradiction as $X\cap Y=\emptyset$. This proves \ref{MH2}.

Finally suppose $v_i, v_j \in Y$ are neighbors. Then $v_i \in N(v_j) \subseteq X$, but also $v_{i+1} \in X$, contradicting \ref{MH1}. This proves \ref{MH3}.
\end{proof}

{\bf Remark 2.} The differences between our setup and Woodall's are the following. In our definition of $Y_0$, we consider all vertices outside of $C$ at once while Woodall considered individual vertices. Then we add (H2) in the definition of hopping paths and require $G$ to have no long good paths.  As we will show in the next subsection, the condition of having no long good paths is implied by our bound on $\sigma_2$.

\subsection{Proof of Theorem~\ref{main}}

Let $n\geq 13$ and let $G$ be an $n$-vertex $2$-connected graph with $\sigma_2(G)\geq \frac{2n-3}{3}$  having no spanning jellyfish as a subgraph. 

\begin{proof}[Proof of Theorem~\ref{main}]Fix an $L$-maximum cycle $C=v_1,\dots, v_c, v_1$. By Lemma~\ref{mainlem1}, $G-C$ contains only singleton components.

\begin{claim}\label{n-2}
$c \leq n-3$.
\end{claim}
\begin{proof} Suppose $c\geq n-2$. Since $G$ is connected and does not contain a spanning jellyfish,
we have $c=n-2$. Let $\{z, z'\} = V(G - C)$. Then $d(z) + d(z') \geq \sigma_2(G)$, and say $d(z) \geq d(z')$. 
 If $N(z) \cup N(z')$ contains no consecutive vertices, then $c \geq 2\sigma_2(G) > n$.  Therefore by symmetry we may assume $v_1 \in N(z)$ and $v_c \in N(z')$. 
If there exists $v_i \in N(z), v_j \in N(z')$ with $j < i \leq j+2$,  then $z, v_1, \ldots, v_j, z', v_c, \ldots, v_i, z$ is a  cycle longer than $C$. 
 We apply Lemma~\ref{pathbound} to $Q = v_1, \ldots, v_c, A = N(z), B = N(z'), t = 2$ and use~\eqref{+2d} to bound $2|A|+|B|$. If $A \cap B \neq \emptyset$ then we can find a spanning jellyfish. Thus we obtain using $n \geq 10$,
\[n-2 = c \geq \min\{2(|A| + |B|) - 3, 2|A| + |B| + t-4 + 1 + 1\} \geq \min\{\lceil \frac{4n-6}{3} - 3 \rceil, n-1\} = n-1.\] \end{proof}

\begin{claim}\label{goodpath}
Every good path in $G$ has at most $c+1$ vertices.
\end{claim}

\begin{proof}
Suppose $P = x_1, \ldots, x_p$ is a good path with $p \geq c+2$. So $x_p, x_{p-1} \notin N(x_1)$ and $x_2\notin N(x_p)$. 
As $d(x_1) + d(x_p) \geq (2n-3)/3$, we may assume $x_1$ is normal. 
If there exists $x_i \in N(x_1)$ and $x_j \in N(x_p)$ with $i > j$, then as in the previous proof, $j - i \geq 3$. 
Apply Lemma~\ref{pathbound} with $Q = x_2, \ldots, x_{p-1}$, $A = N(x_1), B = N(x_p)$, $t =2$ and use~\eqref{+2d}  to obtain
\[n-2 \geq p-2 \geq c \geq \min\{2(|A|+|B|) - 3, 2|A| + |B| + t-4 + \mathbbm{1}_{x_{p-1} \notin A} + \mathbbm{1}_{A \cap B = \emptyset}\} \]\[\geq \max\{\lceil\frac{4n-6}{3} - 3\rceil,  n-1 - 2 + 1 +0\} = n-2,\]
which holds only if $n-2 =c$, contradicting Claim~\ref{n-2}.
\end{proof}

\begin{claim}\label{nolow}
$L \subseteq V(C)$.
\end{claim}

\begin{proof}
Suppose $G-C$ contains a low vertex $y$, and let $z$ be any other vertex in $G-C$.  If for all distinct $v_i \in N(z)$ and $v_j \in N(y)$, $|i-j| \geq 3$, then by Claim~\ref{eq: replace} with $q = 2$,   $n-3 \geq c \geq  d(y) + 2d(z) + 1 \geq n-1 + 1 =n$, a contradiction.

So we may assume without loss of generality that $v_1 \in N(z)$ and $v_{c-\gamma} \in N(y)$ for some $\gamma \in \{0,1\}$. Let $Q= v_1, \ldots, v_{c-\gamma}$. In fact, we must have $\gamma = 1$ otherwise $z,Q,y$ is a good path with $c+2$ vertices contradicting Claim~\ref{goodpath}. If for some $v_i \in N(z), v_j \in N(y)$, $j < i \leq j+2$, then the cycle $C'=z, v_1, \ldots, v_j, y, v_{c-1}, \ldots, v_i, z$ has at least $c$ vertices with equality only if $i=j+2$. In this case, $v_{i-1}$ cannot be adjacent to $y$, so $C'$ has more low vertices than $C$ does, a contradiction. Applying Lemma~\ref{pathbound} with $t=2$ together with~\eqref{+2d} yields the contradiction
\[n-4 \geq |V(Q)| \geq \min\{2(d(z) + d(y)) - 3, 2d(z) + d(y) + 2-4\} = \{\frac{4n-6}{3} - 3, (n-1) -2\} = n-3.\]
\end{proof}
%
%In particular, we may apply Lemma~\ref{hopping} to $G$ and $C$. 
 
 For each $u \in V(G-C)$, let $R(u) =R_C(u):= \{v_i \in V(C): v_{i-1}, v_{i+1} \in N(u)\}$. For each $v_i\in R(u)$, we call the cycle $C'=v_1\dots v_{i-1} u v_{i+1}\dots v_c$ the cycle obtained by {\em swapping} $v_{i}$ in $C$ with $u$.

\begin{claim}\label{switch}
Let $u \in V(G-C)$ and $v\in R_C(u) \setminus L$. Let $C'$ be the cycle obtained by swapping $v$ in $C$ with $u$. Then $X(C) = X(C')$ and $Y(C) = Y(C')$.  
\end{claim}

\begin{proof}
For simplicity, set $Z = V(G-C)$ and $Z' = V(G-C') = Z - \{u\} \cup \{v\}$. 
Then $Z$ and $Z'$ must both be independent sets, and $uv \notin E(G)$ 
  since $v \in R(u)$. 
Recall that $v \in Y_1(C)$. Then
\[X_1(C') = N_{C'}(Z') = N_C(Z'-\{v\}) \cup N_C(v) \subseteq N_C(Y_0(C)) \cup N_C(Y_1(C)) = X_2(C)\enspace \text{ and }\]
 \[Y_1(C') = \{v_i: v_{i-1}, v_{i+1} \in X_1(C')\} \subseteq  \{v_i: v_{i-1}, v_{i+1} \in X_2(C)\} = Y_2(C).\]

Inductively, we obtain that for all $i \geq 1$, $X_i(C') \subseteq X_{i+1}(C)$ and $Y_i(C') \subseteq Y_{i+1}(C)$. Therefore $X(C') \subseteq X(C)$ and $Y(C') \subseteq Y(C)$. But as $u \in R_{C'}(v)$, % (with respect to $C'$)
  we symmetrically obtain $X(C) \subseteq X(C')$ and $Y(C) \subseteq Y(C')$ and equality holds.
\end{proof}

%By Lemma~\ref{hopping}, three conditions \ref{MH1}--\ref{MH3} hold for $X=X(C)$ and $Y=Y(C)$.
 Recall that $Y_0 = V(G-C)$. By Claim~\ref{n-2}, $|Y_0| \geq 3$. 
Our next claim provides a lower bound on $|R(x)|$.
%\begin{claim}
%$|Y_0|\geq 3$.	
%\end{claim}
%\begin{proof}
%If $|Y_0|\leq 1$, then adding an edge between the vertex in $Y_0$ (if necessary) to the cycle $C$ yields a spanning jellyfish, a contradiction.
%If $|Y_0|=2$, then let $Y_0=\{u,v\}$.
%As an element in $N(u)\cap N(v)$ yields a spanning jellyfish, neighborhoods of $u$ and $v$ are disjoint. 
%Furthermore, definition of $X$ ensures $N(u)\cup N(v)\subseteq X$, and no two vertices of $X$ are consecutive in $C$ by \ref{MH1}. Hence we have
%$|C|\geq  2|X| \geq 2|N(u)\cup N(v)|  \geq \frac{4n-2}{3} > n,$
%a contradiction.
%\end{proof}

\begin{claim}
For each $x\in Y_0\setminus L$, we have 
$|R(x)|\geq  \frac{1}{3}(|Y| + |Y_0|-2).$
\end{claim}
\begin{proof}
Fix $x\in Y_0$. 
%We write $v_i^{+2},v_i^+, v_i^-,v_i^{-2}$ to denote the vertices $v_{i+2}, v_{i+1}, v_{i-1}, v_{i-2}$, respectively. 
For $\ast\in \{+,-\}$, let $Y_\ast=Y_\ast(x) = \{ y \in Y\setminus Y_0 : y^\ast \notin N(x) \}.$

As $y\in Y\setminus Y_0$ is in $R(x)$ if and only if both $y^{-},y^{+}$ are in $N(x)$, we know that $Y_+\cup Y_- = Y\setminus(Y_0\cup R(x))$.
By symmetry, assume that $|Y_+|\geq \frac{1}{2}|Y\setminus(Y_0\cup R(x))|$.
Consider 
$$A=\{v,v^+ : v^+ \in R(x)\}, \enspace B=\{v,v^+,v^{+2} : v\in N_C(x), v^+ \notin R(x)\} \text{ and } D=\{ v^{+2}: v\in Y_+\}.$$
The definition of $R(x)$  implies that $A$ and $B$ are disjoint.
If $v\in Y_+$, then  $v^+\in X \setminus N(x)$ by definition. Hence by Lemma~\ref{hopping},  both $v$ and $v^{+2}$ are not in $X$, and none of $v,v^+,v^{+2}$ is in $N(x)$, implying that $D$ is disjoint from $A$ and $B$.
Thus we have 
\begin{align*}
	 |C|&\geq |A|+|B|+|D| \ge 2|R(x)| + 3(d(x)-|R(x)|) + |Y^+| \\
	 &\geq 2|R(x)| + 3(d(x)-|R(x)|) + \frac{1}{2}\left(|Y| - |Y_0| - |R(x)|\right).
\end{align*}

As $|C|=n-|Y_0|$ and $d(x)\geq \frac{n-1}{3}$, this implies 
$|R(x)| \geq \frac{2}{3}(3d(x) -n + \frac{1}{2} |Y| +\frac{1}{2}|Y_0| ) \geq \frac{1}{3}(|Y| +|Y_0| -2 )$
as desired.
\end{proof}

The following claim yields a vertex in $X$ with many neighbors in $Y\setminus L$.

\begin{claim}\label{inequality}
There is a vertex $x\in X$ such that $d_{Y\setminus L}(x)> \frac{2}{3} |Y\setminus L|.$	
\end{claim}
\begin{proof}
Note that $Y$ is an independent set, so at most one vertex in $Y$ belongs to $L$.
By \ref{MH2} and  the definition of normal vertices, we have
$|E(X, Y\setminus L)| \geq \frac{n-1}{3}|Y\setminus L|.$
By \ref{MH1}, we have $|X|\leq \frac{1}{2}|C| \leq \frac{1}{2}(n-|Y_0|)$, so there exists a vertex $x\in X$ with 
$$d_{ Y\setminus L}(x) \geq \frac{(n-1)(|Y\setminus L|)}{3|X|} \geq \frac{2(n-1)|Y\setminus L|}{3(n-|Y_0|)} > \frac{2}{3}|Y\setminus L|.$$
\end{proof}

Now we prove the main theorem.
As all  claims above hold for an arbitrary $L$-maximum cycle $C$, we may choose an $L$-maximum cycle $C$ and a vertex $x\in X(C)$ satisfying 
\begin{enumerate}[label={\rm (Y\arabic*)}]
\item \label{Y1} $d_{Y(C)}(x)> \frac{2}{3}|Y\setminus L|$, and 
\item \label{Y2} $N_{Y_0(C)}(x)$ is maximal among the choices of $L$-maximum cycle $C$ and $x$ satisfying (Y1).

\end{enumerate}

If $N_{Y_0(C)}(x) = Y_0(C)$, then $G$ has a spanning jellyfish. So choose $y\in Y_0(C)$ that is not adjacent to $x$. By Claim~\ref{nolow}, $y \notin L$.
If there is a vertex $z\in R(y)\cap N(x) \setminus L$, then we can swap $y$ with $z$ to obtain an $L$-maximum cycle $C'$. (Note that $z\notin L$, so $C'$ is still $L$-maximum.)
 Then $C'$ satisfies \ref{Y1} as Claim~\ref{switch}
implies $d_{Y(C')}(x)=d_{Y(C)}(x)$. On the other hand, we have $|N_{Y_0(C')}(x)| = |N_{Y_0(C)}(x)|+1$, a contradiction to  \ref{Y2}.

Thus $R(y)$ and $N(x)$ are subsets of $Y$ which possibly intersect at one vertex in $Y\cap L$. 
Moreover, as $y\notin N(x)\cup L$, we have 
$|Y\setminus L|-1  \geq  (|R(y)|- |Y\cap L|) + |N_{Y\setminus L}(x)|$, implying 
$ |Y|-1 \geq |R(y)|+|N_{Y\setminus L}(x)|.$
Using \ref{Y1} and the previous two claims, we obtain
\begin{align*}
|Y|-1 & \geq  |R(y)| + |N_{Y\setminus L}(x)|> 	\frac{1}{3}(|Y|+|Y_0|-2) + \frac{2}{3}|Y\setminus L| \geq |Y| + \frac{1}{3} (|Y_0|-2) - \frac{2}{3}|Y\cap L|.
\end{align*}
However, as $|Y_0| \geq 3$ and $|Y\cap L|\leq 1$, the final term is at least $|Y|-\frac{1}{3}$, a contradiction. \end{proof}

	\section{Existence of spanning brooms}\label{secc}
In this section, we prove Theorem~\ref{broom}.
For this, the following lemma is useful.
\begin{lemma}\label{lem: voss}
If $G$ is a $2$-connected, $n$-vertex, nonhamiltonian graph with $\sigma_2(G) =n-1$, then $n$ is odd and $G$ contains a copy of $K_{k,k+1}$ where $n=2k+1$. 
\end{lemma}

\begin{proof}
By Theorem~\ref{thm: long cycle Ore}, such a graph $G$ contains a cycle $C = v_1,\ldots, v_{n-1},v_1$ of length $n-1$. Choose such a cycle so that the vertex $v$ outside of $C$ has as large degree as possible. Since $G$ is nonhamiltonian, $v$ has no consecutive neighbors in $C$ and therefore $d(v) \leq (n-1)/2$. If $d(v) = (n-1)/2$, then $n$ is odd and by symmetry we may assume $N(v) = \{v_1, v_3, \ldots, v_{n-2}\}$. For each $1\leq k \leq (n-1)/2$, the vertex $v_{2k}$ can be swapped with $v$ to obtain a new cycle $C'$. Thus $v_{2k}$ has no consecutive neighbors in $C'$ and $d(v_{2k}) \geq \sigma_2(G) - d(v) = (n-1)/2$. Then $N(v_{2k}) = N(v)$, and $G$ contains a $K_{(n-1)/2, (n+1)/2}$ with parts $\{v_1, v_3, \ldots, v_{n-2}\} \cup \{v, v_2, v_4,\ldots, v_{n-1}\}$.

Now suppose $d(v) < (n-1)/2$. If $v_{i}, v_{i+2} \in N(v)$ for some $i$, then $v_{i+1} \notin N(v)$. But the choice of $v$ ensures $d(v)\geq d(v_{i+1})$, thus 
$d(v) + d(v_{i+1}) \leq 2d(v) < n-1=\sigma_2(G)$, a contradiction. So each pair of neighbors of $v$ have distance at least $3$ in $C$. Suppose $v_1 \in N(v)$. For all $v_i \in N(v)-\{v_1\}$, we have $v_{i+1},v_{i+2} \notin N(v_2)$, as otherwise we can find either a hamiltonian  cycle, or a $n-1$-cycle in $G-v_{i+1}$ with $d(v_{i+1})\geq \sigma_2(G)-d(v) > d(v)$, a contradiction. So $d(v_2) \leq |V(C)| -|\{v_{i+1},v_{i+2}: v_i \in N(v) - \{v_1\}\}| - |\{v_2\}|$, which implies $d(v_2) + d(v) \leq (n-1) -d(v) +1 <\sigma_2(G)$, a contradiction. This proves the lemma.
\end{proof}

\begin{proof}[Proof of Theorem~\ref{broom}] Let $G$ be a counter-example.
By Theorem~\ref{main}, $G$ has a cut-vertex, say $w$. Let $S_1,\ldots,S_t$ be the vertex sets of all components of $G-w$ and $s_j=|S_j|$ for $1\leq j\leq t$. %with $|W_1|\geq |W_2|\geq\ldots\geq |W_t$, and for $1\leq j\leq t$, let $G_j=G[W_j+w]$.
Then
%\begin{equation}\label{equ1} \mbox{\em 
$d_G(u) \leq s_j\;$ for each $1\leq j\leq t$ and  $u\in S_j$.%}
 %\end{equation}
 
 So, if $t\geq 4$ and $S_1,S_2$ are two smallest sets, then for each $u\in S_1$ and $u'\in S_2$, $d(u)+d(u')\leq  s_1+s_2\leq (n-1)/2<\sigma_2(G)$, a contradiction.
So we have 
\begin{align}\label{eq: three}
 d(x)+d(y)+d(z) \geq \frac{3}{2}\sigma_2(G)\geq n-\frac{3}{2} \enspace \textit{ for all pairwise nonadjacent vertices } x,y,z.
  \end{align}

{\bf Case 1:} $t=3$. 
Choose any $u_j\in S_j$ for each $j$, then they are pairwise nonadjacent, so 
$n-\frac{3}{2}\leq d(u_1)+d(u_2)+d(u_3)\leq s_1+s_2+s_3\leq n-1$. Thus each $u_i$ has degree $s_i$ and adjacent to $w$. Thus $w$ is adjacent to every vertex in $G$, which yields a spanning star.

{\bf Case 2:} $t=2$. If each of $G_1:=G[S_1]$ and  $G_2:=G[S_2]$ has a hamiltonian  cycle, then $G$ has a hamiltonian  path, which is a broom. Suppose $G_1$ has no hamiltonian  cycle. Then by Ore's Theorem, there are non-adjacent $u,u'\in S_1$ with $d_{G_1}(u)+d_{G_1}(u')\leq s_1-1$. So
$$\frac{2n}{3}-1\leq \sigma_2(G)\leq d_{G}(u)+d_{G}(u')\leq d_{G_1}(u)+d_{G_1}(u')+2\leq s_1+1.$$
Therefore, $s_1\geq \frac{2n}{3}-2$ and $s_2=n-1-s_1\leq \frac{n}{3}+1<\frac{2n}{3}-1$. So, $G_2$ does have a hamiltonian  cycle, thus $G$ has a path $P_2$ starting from $w$ and passing through all $S_2$.

Choose a vertex $x\in S_2$.
For every non-adjacent vertices $u,v\in S_1$, \eqref{eq: three} yields that 
\begin{align}\label{eq: sigma3}
d(u)+d(v)\geq \lceil n-\frac{3}{2}-d(x)\rceil  \geq s_1.	
\end{align}
Let $H$ be obtained from $G_1$ by adding $w$ adjacent to all vertices in $S_1$. Since $G_1$ is connected, $H$ is $2$-connected. As $\sigma_2(H)\geq s_1$ by \eqref{eq: sigma3}
, Lemma~\ref{lem: voss} yields that either  $H$ has a hamiltonian  cycle or $s_1=2s$ is even and $H$ contains $K_{s,s+1}$.

{\bf Case 2.1:} $H$ has a hamiltonian  cycle $C=u_1\dots u_{s_1}w u_1$. If one of the edges
$wu_1,wu_{s_1}$ is in $G$, then we have a path $P_1$ starting from $w$ and visiting all vertices in $S_1$. Together with $P_2$, it will form a hamiltonian  path in $G$, a contradiction. If $wu_1,wu_{s_1}\notin E(G)$, then $u_1u_{s_1}\notin E(G)$ as $G_1$ has no hamiltonian  cycles. Thus
 $d_{G}(u_1)+d_{G}(u_{s_1})\leq s_1-1$, a contradiction to \eqref{eq: sigma3}.
 
 {\bf Case 2.2:} $s_1=2s$ is even and $H$ contains $K_{s,s+1}$. Then $G_1$ contains either 
$K_{s,s}$ or $K_{s-1,s+1}$. Since $G_1$ is not hamiltonian, the latter is true. So, since  $w$ has a neighbor in $S_1$, the graph $G[S_1+w]$ contains a spanning broom whose long path starts from $w$. Together with $P_2$, this yields a spanning broom in $G$ as desired.
\end{proof}

\bigskip

{\bf Concluding remarks.}
The provided extremal example for Theorem~\ref{main} (see Fig. 2) is $2$-connected but not $3$-connected. Based on this fact, it is natural to further ask how higher connectivity affects the minimum degree threshold. In other words, what is the minimum $f_k(n)$ such that every $n$-vertex $k$-connected graph contains a spanning jellyfish? In particular, what is $f_3(n)$? Consider a graph $F$ with vertex partition $V_1\cup V_2\cup V_3\cup V_4$ of four sets of sizes $\lceil n/4\rceil \geq |V_1|\geq |V_4|\geq |V_2|\geq |V_3|\geq \lfloor n/4\rfloor$, where $F[V_{i},V_{i+1}]$ is a complete bipartite graph for each $i\in \{1,2,3\}$. As this graph $F$ lacks a spanning jellyfish, $f_k(n) \geq (n-3)/4$ for any $k\leq (n-3)/4$. Determining the precise values of $f_k(n)$ would be an interesting open problem.

Call a graph obtained from a spider by replacing a longest path starting from the branching vertex with a cycle an {\bf octopus}. In particular, an octopus is a subdivision of jellyfish.
Another interesting problem would be to find the minimum $g_k(n)$ such that each $n$-vertex $k$-connected graph $G$ with $\delta(G)\geq g_k(n)$ contains a spanning octopus. If in the previous construction $F$, we have $|V_1|, |V_4| > |V_2|, |V_3|$, then $F$ is also highly connected with minimum degree close to $n/4$ but lacks a spanning octopus. 
Determining $g_k(n)$ for $k \in \{1,2\}$ would already be interesting.

\paragraph{Acknowledgement.} We thank Alexander Sidorenko for attracting our attention to the results in~\cite{keyring1, keyring2, keyring3}.

  \end{document}